\documentclass[oneside]{amsart}
\usepackage[T1]{fontenc}
\usepackage[utf8]{inputenc}
\usepackage[british]{babel}
\usepackage{preamble}
\usepackage{longtable}

\usepackage{amsmath,amsfonts,amsthm,amssymb,mathrsfs,stmaryrd,mathtools,comment,enumitem,soul,booktabs,longtable,tabu,relsize,bbm,geometry,multicol}
\usepackage{todonotes}
\usepackage{hyperref}
\usepackage{cleveref}
\usepackage[usegeometry,paper=8.5in:11in]{typearea}
\KOMAoptions{twoside=false}
\newcommand{\nv}{\nu}

\newcommand{\mmf}{\mathit{mmf}}
\newcommand{\tmf}{\mathit{tmf}}
\newcommand{\ko}{\mathit{ko}}

\definecolor{gray}{gray}{0.5}
\definecolor{cyan}{rgb}{0, 1, 1}
\definecolor{green}{rgb}{0, 0.65, 0}
\definecolor{magenta}{rgb}{1, 0, 1}

\begin{document}

	\title{The $\mathbb C$-motivic Adams-Novikov spectral sequence for topological modular forms}
	
\author[Isaksen]{Daniel C.\ Isaksen}
\address{Department of Mathematics, Wayne State University, Detroit, MI 48009, USA}
\email{isaksen@wayne.edu}
	
		\author[Kong]{Hana Jia Kong}
\address{School of Mathematics, Institute for Advanced Study, NJ, USA}
\email{hana.jia.kong@gmail.com}

	\author[Li]{Guchuan Li}
\address{Department of Mathematics, University of Michigan, MI, USA}\email{guchuan@umich.edu}

  \author[Ruan]{Yangyang Ruan}
\address{Institute of Mathematics,
Academy of Mathematics and Systems Science, Chinese Academy of Sciences, Beijing, China}\email{ruanyy@amss.ac.cn}

	\author[Zhu]{Heyi Zhu}
\address{Department of Mathematics, University of Illinois, Urbana-Champaign, IL, USA}\email{heyizhu2@illinois.edu}

\date{\today}

\thanks{The first author was partially supported by National Science Foundation Grant DMS-2202267. The second author was supported by National Science Foundation grant DMS-1926686. The third author would like to thank the Max Planck Institute for Mathematics and the Hausdorff Research Institute for Mathematics for the hospitality.}

\keywords{topological modular forms, motivic stable homotopy theory, Adams-Novikov spectral sequence, Adams spectral sequence, stable homotopy group}

\subjclass[2010]{Primary 14F42, 55Q10, 55T15; 
  Secondary 55Q45} 

\begin{abstract}
We analyze the $\C$-motivic (and classical) Adams-Novikov spectral sequence
for the $\C$-motivic modular forms spectrum $\mmf$ (and for the classical topological modular forms spectrum $\tmf$).  We primarily use purely algebraic techniques, with a few exceptions.  Along the way, we settle a previously unresolved detail about the multiplicative structure of the homotopy groups of $\tmf$.
\end{abstract}

\maketitle

\section{Introduction}
\label{sec:intro}

The topological modular forms spectrum $\tmf$ plays an essential role
in the study of the stable homotopy groups of spheres
\cite{Bau08} \cite{Beh20} \cite{DFHH14} \cite{Goe10} \cite{Hop95} \cite{Hop02} \cite{HM} \cite{Rezk02}.
The unit map $S \rightarrow \tmf$ from the sphere spectrum to $\tmf$ 
detects much of the structure of the stable homotopy groups
of $S$, including the elements $\eta$ (1-stem), $\nu$ (3-stem),
$\epsilon$ (8-stem), $\kappa$ (14-stem), $\kappabar$ (20-stem),
and many additional elements.  
The unit map is far from injective (for example, $\sigma$ (7-stem) maps to
zero in $\tmf$), so it does not detect all of the stable homotopy groups
of spheres.  Moreover, it is also not surjective.  The computation
of the $\tmf$-Hurewicz image is a difficult problem
that leads to the identification of infinite $v_2$-periodic families
in the stable homotopy groups of spheres \cite{BMQ20}.

The spectrum $\tmf$ serves as an approximation to the sphere spectrum.
This approximation is highly suitable for testing theories and for developing computational techniques.  The structure of $\tmf$ is complicated enough to exhibit the complex phenomena related to the computation of stable homotopy groups, but it is also simple enough to be computed exhaustively.  
We have found that the study of $\tmf$ is an indispensable step along the way to understanding the sphere spectrum.

By comparison, the spectrum $\ko$ is arguably too simple to serve as a test case for computational theories.  For example, its Adams spectral sequence collapses, so its homotopy reduces to an entirely algebraic problem.  Neither the Adams nor the Adams-Novikov spectral sequence collapses for $\tmf$.  However, the analysis of $\tmf$ does not involve crossing differentials or crossing extensions 
in the sense of \cite[Section 2.1]{IWX20}.  This means that the homotopy of $\tmf$ does not share the most delicate parts of the homotopy groups of spheres.

Bruner and Rognes \cite{BR21} have recently exhaustively studied the Adams spectral sequence for $\tmf$.  They completely determine the additive and
(primary) multiplicative structure of the stable homotopy groups of $\tmf$,
with one exception.

The goal of this manuscript is to carry out the Adams-Novikov spectral
sequence for $\tmf$.  In fact, we will work in the more general
$\C$-motivic context and compute the motivic Adams-Novikov spectral sequence
for the $\C$-motivic modular forms spectrum $\mmf$.  The classical computation
is easily recovered from the motivic computation by an algebraic localization.

More specifically, there is a certain motivic element $\tau$.  Inverting
$\tau$ has the effect of collapsing $\C$-motivic computations to
classical computations.  In particular, $\tau$-torsion phenomena disappear
in the classical context.  Henceforth, we will work in the $\C$-motivic context.  The interested reader can easily recover classical computations
from our work by inverting $\tau$.

From another perspective, we also compute the $\C$-motivic effective slice spectral sequence for $\mmf$, since it agrees with the Adams-Novikov spectral sequence over $\C$.
This identification of spectral sequences 
does not appear to be cleanly stated in the literature, but it is a computational consequence of the weight 0 result of \cite[Theorem 1]{Levine15}.

Our goal is not merely to record the details of the Adams-Novikov spectral sequence, which have previously appeared in \cite{Bau08}.  More specifically, we have attempted to give proofs that are as algebraic as possible.  Such algebraic proofs are less likely to contain subtle mistakes, and they are more easily verifiable by machine.  The motivic context provides us with additional algebraic tools that are not accessible in the strictly classical context.  We also correct a few oversights and minor mistakes in the analysis of \cite{Bau08}.

\subsection{Algebraic philosophy}

We do not use any information from the sphere spectrum as input for our computations.  We do, however, assume full knowledge of the
algebraic structure of the motivic Adams and motivic Adams-Novikov $E_2$-pages for $\mmf$.  This is consistent with our goal of using algebraic techniques whenever possible.  It is also consistent with our philosophy that the 
role of $\tmf$ is to inform us about the sphere spectrum.  
By comparison, in \cite{BR21} it is necessary to import the relation
$\eta^2 \kappa = 0$ to $\tmf$ from previous knowledge of the sphere spectrum.  
Fortunately for us, we have the 
relation $h_1^2 d = 0$ in the Adams-Novikov $E_2$-page for $\mmf$.  Because there are no elements in higher filtration, the relation $\eta^2 \kappa = 0$ therefore has an entirely algebraic proof.

A computation involving the Adams or Adams-Novikov spectral sequence breaks into two main stages.  The first stage is entirely algebraic and involves the computation of the $E_2$-page.  In the modern era, this first stage is typically conducted by machine.  The computation of the $E_2$-pages for $\tmf$ is not elementary, but it can be done manually with enough patience
\cite{Baer} \cite[Section 7]{Bau08} \cite{BR21}  \cite[Section 18]{Rezk02}.

The second stage of the process involves the computation of differentials and hidden extensions.  
This stage typically requires input from topology, so it cannot be
fully automated because it is not entirely algebraic.

Our contribution is to recognize that much of this topological second stage actually can be carried out using only algebraic information.  The key idea is to
use the additional structure of the motivic context in order to
pass back and forth between the Adams and Adams-Novikov spectral sequences.  
Each $E_2$-page tells us some things about the homotopy groups of
$\tmf$.
The information contained in these $E_2$-pages does overlap, but not perfectly.  The union of the information in both $E_2$-pages is strictly greater than the information in either one of the $E_2$-pages.

We give several concrete examples of information available in only one
of the two $E_2$-pages.
\begin{enumerate}
\item
In the classical Adams $E_2$-page for $\tmf$, we have the relation $h_1^4 = 0$.  This implies the relation $\eta^4 = 0$ in homotopy.  However, in the classical Adams-Novikov $E_2$-page, the element $h_1^4$ is non-zero and
is hit by an Adams-Novikov $d_3$ differential.  Thus, the relation
$\eta^4 = 0$ has an entirely algebraic proof, but only in the Adams spectral sequence.
\item
In fact, the relation $h_1^4 = 0$ is a consequence of the Massey product
$h_1^2 = \langle h_0, h_1, h_0 \rangle$ in the Adams $E_2$-page.  In the classical Adams-Novikov $E_2$-page, the corresponding Massey product
$\langle 2, h_1, 2 \rangle$ is zero.  Consequently, the Toda bracket
$\eta^2 = \langle 2, \eta, 2 \rangle$ has an entirely algebraic proof, but only in the Adams spectral sequence.
\item
In the classical Adams-Novikov $E_2$-page for $\tmf$, we have the
relation $h_2^3 = h_1 c$.  This implies the relation $\nu^3 = \eta \epsilon$.
However, in the classical Adams $E_2$-page, we have $h_2^3 = 0$.
In fact, there is a hidden $\nu$ extension from $h_2^2$ to $h_1 c$ in the Adams spectral sequence.  Thus, the relation $\nu^3 = \eta \epsilon$
has an entirely algebraic proof, but only in the Adams-Novikov spectral sequence.
\item
In fact, the relation $h_2^3 = h_1 c$ is a consequence of the Massey product
$c = \langle h_2, h_1, h_2 \rangle$ in the Adams-Novikov $E_2$-page.
In the classical Adams $E_2$-page, the corresponding Massey product
is zero.  Consequently, the Toda bracket $\epsilon = \langle \nu, \eta, \nu \rangle$ has an entirely algebraic proof, but only in the Adams-Novikov spectral sequence.  See \cref{lemma:todaepsilon} for more detail on this example.
\end{enumerate}

In order to obtain one key Adams-Novikov differential, 
we use Bruner's theorem on the interaction between algebraic Steenrod operations \cite{May70} and Adams differentials in the context of the Adams spectral sequence.  
We refer to \cite[Theorem 2.2]{Bruner84} for a precise readable statement; 
see also \cite{BMM86} and \cite{Makinen73}.
The practical implementation of Bruner's theorem requires only algebraic information in the form
of algebraic Steenrod operations on $\Ext$ groups.
These operations can be computed by machine, although not as effectively
as the additive and multiplicative structure of the $\Ext$ groups.
The algebraic Steenrod operations are additional structure on top
of what topologists usually think of as ``standard homological algebra".

In the context of the Adams-Novikov spectral sequence, 
we also rely on the Leibniz rule in the form $d_r(x^k) = k x^{k-1} d_r(x)$.
Philosophically, this formula is connected to Bruner's theorem, although we do not know how to make a precise connection.  As in the case of Bruner's theorem, it feels like slightly more information than is usually considered in standard homological algebra.

We also draw attention to \cref{prop:h0ext110}, in which we establish a hidden $2$ extension in the $110$-stem.  Here we use some information about the homotopy groups of $\mmf/\tau^2$.  One might argue that this information is not
entirely of an algebraic nature.
By comparison, the corresponding $2$ extension in the Adams spectral sequence is hidden, but not particularly difficult
\cite[Theorem 9.8(110)]{BR21}.

\subsection{Techniques}

\cref{subsec:hidden} describes a particularly powerful method for studying the $\C$-motivic Adams-Novikov spectral sequence in a way that has no classical analogue.  There is a map
$q: \mmf/\tau \rightarrow \Sigma^{1,-1} \mmf$ that can be viewed as projection to the top cell of the $2$-cell $\mmf$-module $\mmf/\tau$. The homotopy of $\mmf/\tau$ is entirely understood in an algebraic sense since it is isomorphic to the classical Adams-Novikov $E_2$-page for $\tmf$.  Moreover, the map $q$ maps onto the homotopy of $\mmf$ that is annihilated by $\tau$.
Thus $q$ can be used to detect structure in $\mmf$ that is related to classes that are annihilated by $\tau$.

In practice, many specific questions about hidden extensions do not directly involve elements that are annihilated by $\tau$.  Frequently, if we multiply these elements by a power of $\tau$ and a power of $g$, then we end up with elements that are annihilated by $\tau$.  We can use $q$ to understand these latter elements, and finally deduce information about the original elements.  
\cref{hidtaumethod} lists numerous specific examples of this process.
The majority of hidden extensions can be handled very easily in this way,
although a few extensions require more complicated arguments.

We avoid the use of Toda brackets whenever possible, but occasionally they are inevitable.  In those cases where we must compute a Toda bracket, we once again rely exclusively on algebraic techniques.  Namely, our Toda brackets arise from corresponding Massey products in either the Adams or Adams-Novikov $E_2$-page.  The Moss Convergence Theorem \cite{Mos70} says that such algebraic Massey products detect Toda brackets in ``well-behaved" situations.  In practice, all of the situations that we study are well-behaved.

\subsection{The differentials on $\Delta^k$}

Having carried out the entire analysis of the motivic Adams-Novikov spectral sequence for $\mmf$, we can see in hindsight that there are a few key steps from which all of the other miscellaneous computations follow.  Our experience shows that the key steps involve the differentials on elements of the form $2^j \Delta^k$.  This is not particularly surprising; we expect the element $\Delta$ to play a dominant role since it represents $v_2$-periodicity.

First, we establish $d_5(\Delta) = \tau^2 h_2 g$ in \cref{prop:d5_D}.  This follows immediately by comparison to the Adams spectral sequence, in which $\tau^2 h_2 g$ is already zero in the $E_2$-page.  Thus, we have an algebraic proof for $d_5(\Delta)$.  Then the Leibniz rule implies that $d_5(\Delta^2) = 2 \tau^2 \Delta h_2 g$.

The Leibniz rule also implies that $d_5(\Delta^4) = 4 \tau^2 \Delta^3 h_2 g$.
However, $4 \tau^2 \Delta^3 h_2 g$ is zero in the Adams-Novikov $E_2$-page.  Because of the hidden $2$ extension from $2 \tau^2 h_2$ to $\tau^3 h_1^3$, the element
$\tau^3 \Delta^3 h_1^3 g$ ought to play the role of $4 \tau^2 \Delta^3 h_2 g$.
This strongly suggests that there should be a differential
$d_7(\Delta^4) =  \tau^3 \Delta^3 h_1^3 g$.  In fact, this formula is 
correct (see \cref{prop:d7}), but it requires some work to give a precise proof.

Our solution, once again, is to play the Adams and Adams-Novikov spectral sequences against each other.
We used the Adams $E_2$-page to obtain the Adams-Novikov differential $d_5(\Delta)$.  Then we used the Leibniz rule in the Adams-Novikov spectral sequence to obtain $d_5(\Delta^2)$.  In turn, this last Adams-Novikov differential implies an Adams differential $d_2(\Delta^2)$, or
$d_2(w_2)$ in the notation of \cite{BR21}.
Next, we obtain an Adams differential $d_3(\Delta^4)$, or $d_3(w_2^2)$
in the notation of \cite{BR21}, by 
applying Bruner's theorem on the interaction between squaring operations and Adams differentials \cite{BMM86} \cite{Bruner84}.
Finally, the Adams differential $d_3(\Delta^4)$ implies that there is
an Adams-Novikov differential $d_7(\Delta^4)$.
For more details, see \cref{subsctn:d5,subsec:d7}.
Curiously, precise statements about the Adams-Novikov differential 
$d_7(\Delta^4)$ are missing from \cite{Bau08} \cite{HM} \cite{Rezk02}.

\subsection{Main results}

Our main results are expressed in the charts in \cref{sec:charts}.
For completeness, we express this in the form of a main theorem.

\begin{thm}
\label{thm:charts}
The charts in \cref{sec:charts} represent the $\C$-motivic Adams-Novikov spectral sequence for the motivic modular forms spectrum $\mmf$, including 
complete descriptions of 
\begin{itemize}
\item
the $E_2$-page.
\item
all differentials.
\item
the $E_\infty$-page.
\item
all hidden extensions by $2$, $\eta$, and $\nu$.
\end{itemize}
\end{thm}

The proof of \cref{thm:charts} consists of the sum of a long list of miscellaneous computations, which are carried out throughout the manuscript.  See especially the tables in \cref{sec:tables}.  These tables summarize the main computational facts, and they give cross-references to more detailed proofs of each fact.

Our work is not as complete as \cite{BR21} because we have not completely analyzed the multiplicative structure.  In principle, this could be done using the same techniques.
We do study one family of multiplicative relations in more detail.
Bruner and Rognes identify a family $\nu_k$ of elements in the homotopy of
$\tmf$.  They mostly determine the products among these elements, but they
leave one case unresolved.  Our techniques settle this last detail about the $2$-primary multiplicative structure of the homotopy of $\tmf$.

\begin{thm}
\label{nu4-nu6}
In the context of \cite{BR21}, $\nu_4 \nu_6 = \nu \nu_2 M$.
\end{thm}

\cref{nu4-nu6} is proved later as \cref{cor:nu4-nu6}.
In fact, it is a consequence of the more general \cref{nu_i-nu_j}, which 
offers a graceful simultaneous analysis of products $\nu_j \nu_k$.  Bruner and Rognes empirically observed the formula
\[
\nu_i \nu_j = (i+1) \nu \nu_{i+j}.
\]
Our proof shows that the coefficients $(i+1)$ arise naturally from the Leibniz rule 
\[
d_5(\Delta^{i+1}) = (i+1) \Delta^i d_5(\Delta).
\]

\subsection{Future directions}

Our work raises some questions that deserve further study.

\begin{problem}
Compute the $\kappabar$-periodic $\C$-motivic spectrum $\mmf[\kappabar^{-1}]$.
\end{problem}

Frequently, we detect elements and relations by first computing their products with various powers of $g$ or $\kappabar$.  In other words, much of the structure of $\mmf$ is reflected in the $\kappabar$-periodic spectrum $\mmf[\kappabar^{-1}]$.  This motivic spectrum is non-trivial, but its homotopy is entirely annihilated by $\tau^{11}$.  Consequently, its Betti realization is trivial, and it represents purely ``exotic" motivic phenomena.
We mention that \cite{BBC23} also studies $g$-periodic phenomena in $\tmf$, 
although not in a way that is particularly close to our perspective.

\begin{problem}
\label{prob:d7-D^4}
Develop better technology to deduce the differential
$d_7(\Delta^4) = \tau^3 \Delta^3 h_1^3 g$ directly from the differential
$d_5(\Delta) = \tau^2 h_2 g$.
\end{problem}

It is conceivable that $d_7(\Delta^4)$ could be deduced directly from 
$d_5(\Delta)$ using a variant of Bruner's theorem that would apply in the Adams-Novikov spectral sequence,
but we have not even formulated a precise statement of such a variant.
There is a connection between Bruner's theorem and the 
Leibniz rule $d_r(x^2) = 2 x d_r(x)$, but the precise relationship is not clear to us.

Another possible approach to \cref{prob:d7-D^4} might involve an enriched $E_2$-page in which the $2$ extension from $2 \tau^2 h_2$ to $\tau^3 h_1^3$ is not hidden.

\begin{problem}
Construct a spectral sequence whose $E_2$-page reflects the algebraic structure of both the Adams and Adams-Novikov $E_2$-pages.
\end{problem}

We frequently pass back and forth between the Adams and Adams-Novikov spectral sequences.  In order to facilitate these transitions, \cref{sec:ASSANSS} introduces a notion of correspondence between elements of
the Adams spectral sequence and elements of the Adams-Novikov spectral sequence.  

This setup feels like a preliminary attempt to describe a richer connection between the two spectral sequences. 
It would be much more convenient and effective
to compute in just a single spectral sequence that reflects the algebraic structure of both the Adams and Adams-Novikov spectral sequences.  There are some preliminary indications that ``bimotivic
homotopy theory" (also known as $H\F_2$-synthetic $BP$-synthetic homotopy theory) provides a context for this.

\subsection{Outline}

We begin in \cref{sec: background} with a discussion of tools that we will
use to carry out our explicit computations.
We describe both the motivic Adams and motivic Adams-Novikov spectral sequences for $\mmf$, and we establish notation for elements in these spectral sequences.  We also establish notation for certain homotopy elements that we will use later.  We draw particular attention to
\cref{subsctn:inclusion-projection,subsec:hidden},
which establish a powerful tool for detecting hidden extensions.
The basic idea is to use the motivic spectrum $\mmf/\tau$, whose
homotopy is entirely algebraic.

Our explicit computations begin in \cref{sec:diff}, where we establish all
of the Adams-Novikov differentials.  The propositions in this section are mostly
in order of increasing length of differentials.  However, we make some
exceptions to this general rule to preserve the logical order,
so each result only depends on previously proved results.

Once the Adams-Novikov differentials are computed, we proceed to compute
all hidden extensions by $2$, $\eta$, and $\nu$ in \cref{Sec:ext}.
Most of these extensions follow immediately by comparison to the homotopy
of $\mmf/\tau$, but there are several cases with more difficult proofs.

Finally, in \cref{sec:nu_k}, we consider an explicit family of products
that are particularly interesting.  Our results on these products fill a 
gap in the product structure of $\pi_* \tmf$, as described in \cite{BR21}.

\subsection{Conventions}


We work exclusively at the prime $2$.  There are interesting aspects to the
computation of $\tmf$ at the prime $3$ (\cite[Chapter 5]{Bau08}, \cite{DFHH14}, \cite[Chapter 13]{BR21}), but we do not address that topic.
We use the motivic Adams-Novikov spectral sequence to compute the homotopy groups of the $2$-localization of $\mmf$. We also use the $E_2$-page of the motivic Adams spectral sequence, which actually converges to the homotopy groups of the $2$-completion of $\mmf$. 
The distinction between localization and completion is not essential since only finitely generated abelian groups appear in our work.
For expository simplicity, these localizations or completions do not appear in our notation.
For example, the symbol $\mathbb{Z}$ refers to the integers localized at $2$, or to the $2$-adic integers.
Similarly, $\pi_{*,*} \mmf$ refers to the motivic stable homotopy groups
of the $2$-localization (or $2$-completion) of $\mmf$.

The adjective ``motivic" always refers exclusively to the $\C$-motivic context.  We consider no base fields other than $\mathbb{C}$.

Many of our explicit results are labelled with the degrees in which they occur.  These degrees
may help the reader navigate the overall computation, especially in 
finding the relevant elements on Adams-Novikov charts.

\subsection{Acknowledgements}
We thank Tilman Bauer, Robert Bruner, and John Rognes for various discussions related to the production of this manuscript.
We also appreciate stimulating discussions with the participants of the Winter 2023 eCHT reading seminar on the Adams spectral sequence for tmf. 

\section{Background}
\label{sec: background}

In this section, we discuss the techniques that we will use later to carry
out our computations.  

\subsection{The $\mathbb{C}$-motivic modular forms spectrum $\mmf$}
\label{sec:mmf}

There is a 
$\C$-motivic $E_\infty$-ring spectrum $\mmf$ 
that can be viewed as the analogue of the classical topological modular forms spectrum $\tmf$ \cite{GIKR21}.
The Betti realization of $\mmf$ is the classical
spectrum $\tmf$.  More\-over, the cohomology of $\mmf$ is
$A\sslash A(2)$, where $A$ denotes the $\C$-motivic Steenrod algebra 
and $A(2)$ is the subalgebra generated by $\Sq^1$, $\Sq^2$, and $\Sq^4$.

\subsection{The $\C$-motivic Adams spectral sequence for $\mmf$}

We abbreviate the motivic Adams spectral sequence for $\mmf$ by mAss.
The cohomology of $\bbC$-motivic $A(2)$ is the $E_2$-page of the mAss.
The manuscript \cite{Isa09} computes the cohomology of $\bbC$-motivic $A(2)$ using the motivic May spectral sequence, and it gives
a complete description of its ring structure. 
The mAss $E_2$-page consists entirely of algebraic information, which we
take as given.
We grade the mAss $E_2$-page in the form $(s,f,w)$, where $s$ is the topological stem, $f$ is the Adams filtration, and $w$ is the motivic weight.

The motivic Adams differentials are recorded in \cite{Isa18}.
However, this manuscript does not depend on previous knowledge of any
Adams differentials, neither classical nor motivic.
For completeness, we provide self-contained 
proofs for two Adams differentials in \cref{Adams-diff}.

We adopt the notation of \cite{Isa09} and \cite{Isa18} for the mAss.
For the reader's convenience, 
\cref{tab:mAss-notation} provides a concordance between our notation
and the notation of \cite{BR21}.
Beware that the motivic generators $u$ and $\Delta u$
have no classical counterparts because they are annhilated by $\tau$.

\begin{longtable}{lll}
\caption{Generators of the motivic Adams $E_2$-page for $\mmf$ \label{tab:mAss-notation} } \\
\toprule
$(s,f,w)$ & \cite{Isa09}  & \cite{BR21} \\
\midrule \endhead
\bottomrule \endfoot
$(0,1,0)$ & $h_0$ & $h_0$ \\
$(1,1,1)$ & $h_1$ & $h_1$ \\
$(3,1,2)$ & $h_2$ & $h_2$ \\
$(8,3,5)$ & $c$ & $c_0$ \\
$(8,4,4)$ & $P$ & $w_1$ \\
$(11,3,7)$ & $u$ & \\
$(12,3,6)$ & $a$ or $\alpha$ & $\alpha$ \\
$(14,4,8)$ & $d$ & $d_0$ \\
$(15,3,8)$ & $n$ or $\nu$ & $\beta$ \\
$(17,4,10)$ & $e$ & $e_0$ \\
$(20,4,12)$ & $g$ & $g$ \\
$(25,5,13)$ & $\Delta h_1$ & $\gamma$ \\
$(32,7,17)$ & $\Delta c$ & $\delta$ \\
$(35,7,19)$ & $\Delta u$ & \\
$(48,8,24)$ & $\Delta^2$ & $w_2$ \\
\end{longtable}

\subsection{The $\mathbb C$-motivic Adams-Novikov spectral sequence for $\mmf$}
\label{sec: ANSS}
	
The $E_2$-page of the classical Adams-Novikov spectral sequence for $\tmf$ is given by $\Ext^{**}_{{BP}_{*}{BP}}({BP}_{*},{BP}_{*} \tmf)$, where $BP$ denotes the Brown-Peterson spectrum.
Analogously to the classical Adams-Novikov spectral sequence, one can construct a motivic Adams-Novikov spectral sequence by 
re\-solv\-ing with respect to the motivic Brown-Peterson spectrum.
We abbreviate the motivic Adams-Novikov spectral sequence by mANss.
We grade the mANss $E_2$-page in the form $(s,f,w)$, where
$s$ is the topological stem, $f$ is the Adams-Novikov filtration, and $w$ is the motivic weight.

The mANss is easy to describe in classical terms. The motivic $E_2$-page can be obtained from its classical analogue by first assigning a third degree, called the weight, to be half of the total degree for each class, then adjoining a polynomial generator $\tau$ of degree $(0,0,-1)$ (see, e.g. \cite{HKO11}\cite{Isa19}). 
More explicitly, a classical element $x$ in degree $(s,f)$ corresponds to a family of elements $\{\tau^n x | n\geq 0\}$ in the mANss, where the motivic element $x$ has degree $\left( s, f, \frac{s+f}{2} \right)$.

The $E_2$-page of the mANss consists entirely of algebraic information,
which we take as given.
For our purposes, the best way to compute this 
$E_2$-page is by the algebraic Novikov spectral sequence, which
is worked out in detail in \cite{Baer}.

\begin{rem}
The $E_2$-page of the classical Adams-Novikov spectral sequence for $\tmf$ is the cohomology of a version of the elliptic curve Hopf algebroid (\cite{Rezk02}\cite{Bau08}). By the change-of-rings theorem \cite[Theorem 15.3]{Rezk02}, this is the same as the 
cohomology of the Hopf algebroid $(BP_*tmf, BP_*BP \otimes_{BP_*}BP_*tmf)$. See \cite[Proposition ~15.7 and Section ~20]{Rezk02} for more details.
We do not rely on this perspective.
\end{rem}

\subsection{Notation for the motivic Adams-Novikov spectral sequence}

\Cref{generators} lists the multiplicative
generators for the mANss $E_2$-page for $\mmf$.  These generators
are the starting point of our computation.

\begin{longtable}{ll}
\caption{Generators of the motivic Adams-Novikov $E_2$-page for $\mmf$
\label{generators} } \\
\toprule
$(s,f,w)$ & generator \\
\midrule \endfirsthead
\bottomrule \endfoot
$(0, 0, -1)$   & $\tau$ \\
$(1, 1, 1)$   & $h_1$\\
$(3, 1, 2)$   & $h_2$\\
$(5, 1, 3)$   & $h_1v_1^2$ \\
$(8, 0, 4)$   & $P$\\
$(8, 2, 5)$   & $c$\\
$(12, 0, 6)$  & $4a$\\
$(14, 2,8)$   & $d$\\
$(20, 4, 12)$  & $g$ \\
$(24, 0, 12)$ & $\Delta$\\
\end{longtable}

One must be slightly careful with the definitions of some of these
elements because they belong to cyclic groups of order greater than
$2$.  In these cases, there is more than one possible generator.
Specifically, this issue arises for the elements
$h_2$, $P$, $4a$, $g$, and $\Delta$.
For $P$, $4a$, and $g$, we simply choose arbitrary generators.

\begin{rem}
\label{h2-choice}
$(3, 1, 2)$
The choice of $h_2$ makes little practical difference
to us, as long as it is a generator of the mANss $E_2$-page in
degree $(3,1,2)$.  
For definiteness, we take $h_2$ to represent the homotopy element
$\nu$, assuming an a priori definition of $\nu$ (for example,
by appealing to the homotopy of the sphere spectrum or by
appealing to a geometric construction of $\nu$ involving quaternionic multiplication).
\end{rem}

The choice of $\Delta$ also makes little practical difference.
We choose $\Delta$ in such a way to make our formulas easier
to write.  See \cref{Delta-defn} and \cref{nu-defn} for 
more details.  Note that the choice of $\Delta$ depends
on a previous choice of $h_2$.

\begin{rem}
$(12, 0, 6)$
The notation $4a$ does not appear to be natural and deserves some explanation.
There are two closely related reasons why we find this notation to be
convenient.  First, the element $4a$ is detected 
in the algebraic Novikov spectral sequence \cite{Baer} by an element
$h_0^2 a$.  Second, the element $2 \cdot 4a$ turns out to be a 
permanent cycle that detects an element in $\pi_{12,6} \mmf$.
This same homotopy element is detected by $h_0^3 a$ in
the Adams spectral sequence for $\mmf$.
\end{rem}

The element $g$ is a permanent cycle and therefore
represents a homotopy class $\kappabar$.
Multiplication by $g$ provides regular structure to the mANss for $\mmf$.
We typically sort elements into families that are related by $g$ multiplication.
In other words, when we consider a particular element $x$, we also
typically consider the elements $x g^k$ for all $k \geq 0$ at the same time.

Taken together, \cref{fig:E2v1,fig:E2d7} depict the $E_2$-page of the mANss for $\mmf$ graphically. 
The careful reader should
superimpose these figures in order to obtain a full
picture of the mANss.
\Cref{fig:E2v1} depicts a regular 
$v_1$-periodic
pattern in the $E_2$-page, to be discussed in detail in
\Cref{sec:bo}.
\Cref{fig:E2d7} depicts the remaining classes.

\subsection{Comparison between the mANss and the mAss}
\label{sec:ASSANSS}

\begin{defn}
Let $a$ be a permanent cycle in the mANss for $\mmf$, and let
$b$ be a permanent cycle in the mAss for $\mmf$.
The elements $a$ and $b$ \textbf{correspond}
if there exists a non-zero element in $\pi_{*,*} \mmf$ that is detected
by $a$ in the mANss for $\mmf$ and is detected by $b$
in the mAss for $\mmf$.
\end{defn}

\begin{rem}
\label{rem:higherfiltration}
Beware that a permanent cycle may detect more than one element in
$\pi_{*,*} \mmf$, depending on the presence of permanent cycles in
higher filtration.  We ask only that the cosets detected by $a$
and $b$ intersect; they need not coincide.
We give an explicit example.

The element $P$ of the mANss $E_\infty$-page detects two elements 
of $\pi_{8,4} \mmf$ because of the presence of $\tau c$ in higher
filtration.
On the other hand, the element $P$ of the mAss $E_\infty$-page
detects infinitely many elements (which differ only by a 
$2$-adic unit factor) because of the presence of
$P h_0^k$ in higher filtration for $k \geq 1$.
This is an example of a corresponding pair of elements that do
not detect precisely the same coset of homotopy elements.
\end{rem}

\begin{rem}
It is possible that a single element of the mANss corresponds to two
different elements of the mAss.  For example, the element
$P$ of the mANss detects two elements 
of $\pi_{8,4} \mmf$ because of the presence of $\tau c$ in higher
filtration.
These two homotopy elements are detected by $\tau c$ and by $P$
in the mAss.  Consequently, the mANss element $P$ corresponds to
the mAss element $P$, and it also corresponds to the mAss element $\tau c$.
Fortunately, this kind of complication never arises for us
in practice.  For example, none of the correspondences
listed in \cref{table:ASSANSS} exhibit this type of behavior.
\end{rem}

\begin{rem}
The element $2$ of the mANss $E_\infty$-page detects a single
element in homotopy since there are no elements in higher filtration.  On the other hand, the element $h_0$ of the mAss
$E_\infty$-page detects infinitely many elements in homotopy,
all of which differ by a $2$-adic unit factor, because
of the presence of $h_0^k$ in higher filtration.
Consequently, while $2$ and $h_0$ are a corresponding pair,
they do not detect the same sets of homotopy elements.
Rather, the homotopy elements detected by $2$ form a subset
of the homotopy elements detected by $h_0$.

Among the corresponding pairs listed in \cref{table:ASSANSS}, the same phenomenon occurs for $h_2$, $g$, $\Delta h_1$, and $4 \Delta^2$.  In all of these cases, the homotopy elements detected
by the mANss $E_\infty$-page element form a subset of the
homotopy elements detected by the mAss $E_\infty$-page element.
\end{rem}

Multiplicative structure respects corresponding pairs.
The following proposition establishes
this principle precisely.

\begin{prop}
\label{prop:correspond-product}
Let $a$ and $a'$ be elements of the mANss $E_\infty$-page, and let
$b$ and $b'$ be elements of the mAss $E_\infty$-page.
If $a$ corresponds to $a'$, $b$ corresponds to $b'$, 
and $a b$ and $a' b'$ are non-zero; then
$a b$ corresponds to $a' b'$.
\end{prop}

\begin{proof}
Let $a$ and $a'$ detect a homotopy element $\alpha$, and
let $b$ and $b'$ detect a homotopy element $\beta$.
Then $a b$ and $a' b'$ detect the product $\alpha \beta$.
\end{proof}

\begin{rem}
\label{rem:Thom-reduction}
The motivic Thom reduction map $BP \rightarrow H\F_2$
induces a map from the mANss for $\mmf$ to the mAss for $\mmf$.
This map detects some corresponding pairs but not all of them.
Namely, it detects the pairs involving $h_1$, $h_2$, and $g$.
These are the elements for which there is no filtration shift
between the mANss and the mAss.
\end{rem}

\subsection{Homotopy elements}
\label{subsctn:homotopy-elements}

\cref{htpy-name} lists some notation that we use
for elements in the homotopy of $\mmf$.
We use the same symbols as in \cite{BR21} for our motivic versions.
Beware that some of our homotopy elements may not be exactly
compatible under Betti realization with the ones in \cite{BR21}.
We discuss the details
of these ambiguities in the following paragraphs.

We define elements in homotopy by specifying the elements in the mANss $E_\infty$-page that detect them.
In some cases, it is already easy to see that these detecting
elements survive to the $E_\infty$-page.
For example, there are no possible targets for differentials on
$h_1$ and $h_2$; nor can they be hit by differentials.
Beware that we do not yet know that some of 
these detecting elements actually survive to the $E_\infty$-page.
This will only become apparent after our analysis of Adams-Novikov
differentials.

In some cases, there are  $E_\infty$-page elements in higher filtration.
When this occurs, the specified element in the $E_\infty$-page
detects more than one element in homotopy. 
For example, the element $\tau h_1^3$ lies in filtration higher than
the filtration of $h_2$.  Therefore, $h_2$ detects two distinct
elements in homotopy.
In \cref{htpy-name}, this ambiguity occurs only for $\nu$,
$\kappa_4$, and the elements of the form $\nu_k$.

The choice of $\nu$ is of little practical signficance to us.
For definiteness, we may use an a priori definition of $\nu$,
as discussed in \cref{h2-choice}.
The choices of $\nu_k$ will be discussed later in
\cref{defn:nu_i}.  The choice of $\kappa_4$ is immaterial
for our purposes, so it can be an arbitrary
generator of $\pi_{110,56}$.

\begin{longtable}{lll}
\caption{Some elememts of  $\pi_{*,*} \mmf$ \label{htpy-name} } \\
\toprule
$(s,w)$ & name & detected by \\
\midrule \endhead
\bottomrule \endfoot
$(1,1)$ & $\eta$ & $h_1$ \\
$(3,2)$ & $\nu$ & $h_2$ \\
$(8,5)$ & $\epsilon$ & $c$ \\
$(14,8)$ & $\kappa$ & $d$ \\
$(20,12)$ & $\bar\kappa$ & $g$ \\
$(25,13)$ & $\eta_1$ & $\Delta h_1$ \\
$(27,14)$ & $\nu_1$ & $2\Delta h_2$ \\
$(51,26)$ & $\nu_2$ & $\Delta^2 h_2$ \\
$(96,48)$ & $D_4$ & $2 \Delta^4$ \\
$(99,50)$ & $\nu_4$ & $\Delta^4 h_2$ \\
$(110,56)$ & $\kappa_4$ & $\Delta^4 d$ \\
$(123,62)$ & $\nu_5$ & $2\Delta^5 h_2$ \\
$(147,74)$ & $\nu_6$ & $\Delta^6 h_2$ \\
$(192,96)$ & $M$ & $\Delta^8$ \\
\end{longtable}

\begin{rem}
$(20,4,12)$
Bruner and Rognes choose $\kappabar$ by reference to the
unit map $S \rightarrow \tmf$, together with a prior
choice of $\kappabar$ in $\pi_{20} S$.
For our purposes, we only need that $\kappabar$ is detected by
$g$ in the mANss $E_\infty$-page, so we may choose $\kappabar$
to be compatible with the one in \cite{BR21}.

There is a slight complication with $\kappabar$.
In \cite{Isa19} and \cite{IWX20}, the symbol $\kappabar$ is used
for an element of $\pi_{20,11} S^{0,0}$ that is detected by
$\tau g$ in the motivic Adams spectral sequence.
The point is that $g$ does not survive the May spectral sequence,
so it does not exist in the motivic Adams spectral sequence.

Here, we use $\kappabar$ for an element of $\pi_{20,12} \mmf$.
This element is detected by $g$ in the Adams spectral sequence for
$\mmf$.  The unit map $S^{0,0} \to \mmf$ takes
$\kappabar$ to $\tau \kappabar$.
\end{rem}

\begin{rem}
\label{edge-homomorphism}
Bruner and Rognes refer to the ``edge homomorphism" in order to specify
certain elements in $\pi_* \tmf$.  From the perspective of the
Adams-Novikov spectral sequence, this edge homomorphism takes a particularly
convenient form that can be easily described as a surjection followed
by an injection. 
The surjection takes $\pi_{*} \tmf$ onto its quotient
by elements that are detected in strictly positive Adams-Novikov filtration.
In other words, the surjection maps $\pi_{*} \tmf$ onto the
Adams-Novikov $E_\infty$-page in filtration $0$.  Then the injection is the inclusion
of the Adams-Novikov $E_\infty$-page into the Adams-Novikov $E_2$-page in filtration $0$.
In other words, the edge homomorphism detects the homotopy elements
that are detected in Adams-Novikov filtration $0$.
This description of the edge homomorphism applies equally well in the
setting of $\pi_{*,*} \mmf$ and the motivic Adams-Novikov spectral sequence.

The edge homomorphism depends on the choice of $\Delta$ 
(see \cref{Delta-defn}).  
Beware that our choice of $\Delta$ does not guarantee that our
edge homomorphism is identical to the one discussed in \cite{BR21}.
Consequently, our definitions of the 
homotopy elements $D_4$ and $M$
in \cref{htpy-name}
may not be the same as 
\cite[Definition 9.22]{BR21}. 
All possible choices of $\Delta$ differ by multiples of $2$,
so $\Delta^k$ is well-defined up to multiples of $2^k$.
Therefore, our choices of 
$D_4$ and $M$ agree with the Bruner-Rognes 
definitions up to multiples of $16$ and $256$ 
respectively.
\end{rem}

\subsection{$v_1$-periodicity}
\label{sec:bo}

Part of the mANss for $\mmf$ reflects $v_1$-periodic homotopy.
The pattern of differentials in this part is similar to the 
Adams-Novikov differentials for $\ko$ (see \cite[page~31]{Bau08}).
We consider this part separately and omit them from computations of higher differentials.
Beware that we are not employing an intrinsic definition of $v_1$-periodic
homotopy.  Rather, we are simply observing some specific structure
in the mANss for $\mmf$.

In the mANss $E_2$-page, 
consider elements of the form
$\tau^a h_1^b P^m (4 a)^\epsilon \Delta^n$, 
where $\epsilon$ equals $0$ or $1$ and $m + \epsilon > 0$.
We refer to these elements as the $v_1$-periodic classes.

Note that $1$ and $\Delta^n$ (as well as their $\tau$ multiples and $h_1$ multiples) are excluded from this family of elements. 
The knowledgeable reader may observe that
these powers of $\Delta$ satisfy an intrinsic definition of $v_1$-periodicity.
Our family is constructed for its practical convenience, not for its 
intrinsic properties.
The $v_1$-periodic elements, as we have defined them, only interact with
each other through the Adams-Novikov differentials.  However, the powers
of $\Delta$ support Adams-Novikov differentials that take values outside
of the $v_1$-periodic family.  Consequently, we consider them in conjunction
with the non-$v_1$-periodic elements.

\cref{fig:E2v1,fig:E4v1} display the $v_1$-periodic portions
of the mANss $E_2$-pages and $E_\infty$-pages respectively.
Our other charts exclude the $v_1$-periodic family.

\subsection{The spectrum $\mmf/\tau$}

Consider the cofiber sequence
\begin{equation}
\label{eqn:cofta}
\Sigma^{0,-1}\mmf\xrightarrow{\tau} \mmf \xrightarrow{i} \mmf/\tau \xrightarrow{q} \Sigma^{1,-1}\mmf
\end{equation}
of $\mmf$-modules.  The spectrum
$\mmf/\ta$ is a 2-cell $\mmf$-module, in the sense that it is built from
two copies of $\mmf$.  We refer to $i$ as inclusion of the bottom cell,
and we refer to $q$ as projection to the top cell.

The mANss for $\mmf/\ta$ has a particularly simple algebraic form.
The $E_2$-page is isomorphic to the $E_2$-page of the classical
Adams-Novikov spectral sequence for $\tmf$, except that it has a third degree.  However, this additional degree 
carries no extra information since it equals half of the total degree, i.e., the sum of the stem and the Adams-Novikov filtration.

Moreover, the mANss for $\mmf/\ta$ collapses.  There are no differentials, so the
$E_\infty$-page equals the $E_2$-page.  Even better, there are no possible
hidden extensions for degree reasons.  Consequently, the homotopy of
$\mmf/\ta$ is isomorphic to the classical Adams-Novikov $E_2$-page for $\tmf$.
Therefore, we take the homotopy of $\mmf/\tau$ as given since 
it is entirely algebraic information.
The results discussed in this paragraph are $\tmf$ versions of the results
in \cite[Section 6.2]{Isa19}, which are stated for the sphere spectrum.

We use
the notation of \cref{generators} in order to describe homotopy
elements in $\pi_{*,*} \mmf/\tau$.
On the other hand, we need to be more careful about notation
for elements in $\pi_{*,*} \mmf$.  We can specify elements
in $\pi_{*,*} \mmf$ by giving detecting elements in the 
mANss $E_\infty$-page, but this only specifies homotopy elements
up to higher filtration.  See \cref{subsctn:homotopy-elements} for more discussion of choices of elements in $\pi_{*,*} \mmf$.

The mAss for $\mmf/\tau$ is isomorphic
to the algebraic Novikov spectral sequence, for which we have complete 
information \cite{Baer}.  This is a $\tmf$ version of the results
in \cite{BWX21}, which are stated for the sphere spectrum.

\subsection{Inclusion and projection}
\label{subsctn:inclusion-projection}

We discuss the inclusion $i$ and the projection $q$ from \cref{eqn:cofta}
in more detail.  
Many of these ideas first appeared in \cite[Chapter 5]{Isa19} in more
primitive forms.

We already observed that both $i$ and $q$ are $\mmf$-module maps.
Note that the inclusion $i$ is a ring map, but the projection
$q$ is not.  They 
induce maps of motivic Adams-Novikov spectral
sequences.  These spectral sequence maps are in fact module maps 
over the mANss for $\mmf$.  Similarly, the induced 
maps of homotopy groups are $\pi_{*,*}\mmf$-module maps.

We describe the inclusion $i: \mmf \rightarrow \mmf/\tau$
of the bottom cell in computational terms.
If $\alpha$ is a homotopy element that is not a multiple of $\tau$,
then
$i(\alpha)$ is an element of the mANss $E_2$-page that 
detects $\alpha$.
On the other hand, if $\alpha$ is a multiple of $\tau$,
then $i(\alpha)$ is zero.
This fact is closely related to the observation that the
motivic Adams-Novikov spectral sequence is the same as the 
$\tau$-Bockstein spectral sequence.

\cref{htpy-name} gives a number of values of $i$.
For example, we have $i(\eta) = h_1$.  In fact,
we have defined the elements in the middle column of the table
to have the appropriate values under $i$.

For later use, we describe the computational implication that
$q: \mmf/\tau \rightarrow \Sigma^{1,-1} \mmf$ is an $\mmf$-module map.
Let $\alpha$ be an element of $\pi_{*,*} \mmf$, and let
$x$ be an element of $\pi_{*,*} \mmf/\tau$.
The object $\mmf/\tau$ is a right $\mmf$-module, and
\[
x \cdot \alpha = x \cdot i(\alpha),
\]
where the dot on the left side represents the module action and the
dot on the right side represents the multiplication of the ring
spectrum $\mmf/\tau$.
Then we have that
\begin{equation}
\label{q-module-map}
q(x) \cdot \alpha = q(x \cdot \alpha) = q(x \cdot i(\alpha)),
\end{equation}
where the dot on the left represents multiplication in $\mmf$;
the dot in the center represents the $\mmf$-module action
on $\mmf/\tau$; and the dot on the right represents multiplication
in $\mmf/\tau$.

We need a precise statement about the values of $q$. 
Our desired statement has essentially the same content as
\cite[Theorem 9.19(1c)]{BHS19},
which we reformulate into a form that is more convenient for us.

\begin{prop}
\label{topcell}
Let $x$ be an element of the mANss $E_2$-page that is not divisible by $\tau$, 
and suppose that there is a non-zero motivic Adams-Novikov differential
$d_{2r+1}(x) = \tau^r y$. 
If we consider $x$ as an element of $\pi_{*,*} \mmf/\tau$, then
the element $q(x)$ 
of $\pi_{*,*} \mmf$ is detected by $-\tau^{r-1} y$
in the mANss $E_\infty$-page.
\end{prop}

\begin{proof}
The proof is a chase of the right side of the diagram
\[
\xymatrix{
\mmf/\tau \ar[r]^-{\tau^r} & \mmf/\tau^{r+1} \ar[r] & \mmf/\tau^r \ar[r]^{\beta} & \mmf/\tau \\
\mmf \ar[u]^i \ar[d]_{\tau^{r-1}} \ar[r]^{\tau^r} & \mmf \ar[u] \ar[r] \ar[d]_{=} &
\mmf/\tau^r \ar[u]^{=} \ar[r] \ar[d] & \mmf \ar[u]_i \ar[d]^{\tau^{r-1}} \\
\mmf \ar[r]_{\tau} & \mmf \ar[r]_-i & \mmf/\tau \ar[r]_q & \mmf,
}
\]
in which the rows are cofiber sequences.
We start with the element $x$ in $\pi_{*,*} \mmf/\tau$ in the bottom row.
This element lifts to $\mmf/\tau^r$ in the middle row by
\cite[Theorem 9.19]{BHS19} because $x$ survives to the $E_{2r+1}$-page.
The map $\beta$ is the ``Bockstein" mentioned in
\cite[Theorem 9.19]{BHS19}, so we have that $\beta(x)$ equals 
$-y$ in the upper right corner of the diagram.
Then $-y$ lifts to an element of $\pi_{*,*} \mmf$ in the middle row that is detected by $-y$.  Finally, multiply by $\tau^{r-1}$ to obtain
$q(x)$.
\end{proof}

\begin{rem}
\label{rem:topcell}
\cref{topcell} requires that $x$ supports a non-zero Adams-Novikov differential.  On the other hand, suppose that $x$ is a permanent cycle.
Then $x$ is in the image of $i$, and $q(x) = 0$ since 
the composition $q i$ is zero.
\end{rem}

\subsection{Hidden extensions}	
\label{subsec:hidden}

	We briefly review the notion of hidden extensions in spectral sequences.
We adopt the following definition of hidden extensions.

\begin{defn}{\cite[Definition ~4.1.2]{Isa19}}
\label{def:hiddenext}
Let $\alpha$ be an element in the target of a multiplicative spectral sequence, and suppose that $\alpha$ is detected by an element
$a$ in the $E_\infty$-page of the spectral sequence. A hidden extension by
$\alpha$ is a pair of elements $b$ and $c$ of the $E_\infty$-page such that:
\begin{enumerate}
  \item the product $a\cdot b$ equals zero in the $E_\infty$-page.
  \item the element $b$ detects an element $\beta$ in the target such that $c$ detects the product $\alpha \cdot \beta$.
  \item if there exists an element $\beta'$ of the target that is detected by $b'$ such that $\alpha\cdot \beta'$ is detected by $c$, then the filtration of $b'$ is less than or equal to the filtration of $b$.
\end{enumerate}	
\end{defn}

We will use projection $q$ to simplify our analysis of hidden extensions.
We shall show that two different products in $\pi_{*,*} \mmf$
are the image of the same
element in $\pi_{*,*} \mmf/\tau$.  Therefore, they are equal.

\begin{method*}\label{method:ctau}\rm
Suppose that $\alpha$ is not divisible by $\tau$,
so $i(\alpha) = a$, where $a$ is an element of the mANss that detects $\alpha$.
Consider a possible
hidden $\alpha$ extension from $b$ to $c$ in the mANss for $\mmf$.
If $b$ and $c$ detect classes $\beta$ and $\gamma$ that are annihilated by $\tau$,
then $\beta$ and $\gamma$ are in the image of projection $q$ to the top cell.  Let
$\overline{b}$ and $\overline{c}$ be their pre-images in
$\pi_{*,*}(\mmf/\ta)$.  Since this latter object is algebraic and completely
known, we can determine whether $\overline{b}$ and $\overline{c}$
are related by an extension by mere inspection.

\cref{q-module-map} shows that
\[
q( \overline{b} \cdot a ) = q(\overline{b} \cdot i(\alpha) ) =
q(\overline{b}) \cdot \alpha = \beta \cdot \alpha,
\]
where the first two dots represent multiplication in $\mmf/\tau$,
while the last two dots represent multiplication in $\mmf$.
If $\overline{b} \cdot a$ equals $\overline{c}$, then
$\beta \cdot \alpha$ equals $q(\overline{c}) = \gamma$,
and there is a hidden $\alpha$ extension from $b$ to $c$. 

On the other hand, if $\overline{b} \cdot a$ equals zero,
then 
$\beta \cdot \alpha$ equals zero, and there is not a hidden
$\alpha$ extension from $b$ to $c$.
\end{method*}

In practice, \cref{method:ctau} is very effective for determining hidden
extensions.  The main restriction is that it only applies to extensions
between classes that are annihilated by $\tau$.

\begin{example}
\label{exmp:hidextn}	
$(54, 2, 28)$
We illustrate \cref{method:ctau} with a concrete example of 
the hidden $2$ extension from $\De^2 h_2^2$ to $\tau^4 d g^2$ 
in the 54-stem.
In this example, we assume some knowledge
of the relevant Adams-Novikov differentials (see \cref{sec:diff}).
Consequently, one should view this example as a deduction of a
hidden extension from previously determined differentials.

First, multiply by $\tau g$.  If we establish a hidden $2$
extension from $\tau \De^2 h_2^2 g$ to $\tau^5 d g^3$ in the 74-stem, 
then we can immediately conclude the desired extension in the 54-stem.
This step already requires motivic technology, since
both $\De^2 h_2^2 g$ and $d g^3$ are hit by classical Adams-Novikov
differentials.

The key point is that the two elements under consideration in the 74-stem
are non-zero but annihilated by $\tau$.
They are annihilated by $\tau$ because of the differentials
$d_5(\De^3h_2)=\tau^2\De^2h_2^2 g$ and $d_{13}(2\De^3h_2)=\tau^6dg^3$, to
be proved later in Propositions \ref{prop:d5_D} and \ref{prop:d13-2D^3h2}.

The elements $\tau\De^2h_2^2 g$ and $\tau^5dg^3$ 
represent classes in $\pi_{74,39} \mmf$ that are annihilated by $\tau$.
Therefore, these elements lie in the image of
$q: \pi_{75,38} \mmf/\ta \rightarrow \pi_{74,39}\mmf$.

By \cref{topcell}, the preimages in $\pi_{75,38} \mmf/\ta$
are $\De^3 h_2$ and $2 \De^3 h_2$ respectively.
These two elements are connected by a $2$ extension.
Therefore, their images under $q$ are also connected by
a $2$ extension.
\end{example}

\subsection{Toda brackets}\label{subsec:toda}
For background on Massey products and Toda brackets,
including statements of the May convergence theorem and the Moss convergence theorem, we refer readers to \cite{Toda62}, \cite{May69}, \cite{Mos70} and also \cite{Isa19}, \cite{BK21}.

Massey products in the $E_2$-page of an Adams or Adams-Novikov spectral
sequence are algebraic information since they are part of the structure of Ext groups.
Some Toda brackets in homotopy can be deduced directly
from these Massey products using the Moss convergence theorem.
In order to apply this theorem, one must establish the
absence of crossing differentials.  
Whenever we apply the Moss convergence theorem, 
there will be no possible crossing differentials.
In other words, the crossing differentials condition
is satisfied for algebraic reasons.
Thus, the Toda brackets that we use 
are algebraic in the sense that they can be deduced
directly from the algebraic structure of Ext. 

\begin{rem}
In general, Massey products and Toda brackets are defined as sets, not elements. An equality of the form $\langle \alpha, \beta, \gamma \rangle = \delta$ means that
\begin{enumerate}
    \item $\delta$ is contained in the bracket;
    \item the bracket has zero indeterminacy.
\end{enumerate}
\end{rem}

The following lemma gives an explicit example of an
algebraic deduction of a Toda bracket.  See 
\cref{htpy-name} for an explanation of the notation.

\begin{lem}\label{lemma:todaepsilon}
$(8,3,5)$
The Toda bracket $\langle \nu, \eta, \nu \rangle$ 
in $\pi_{8,5}\mmf$ is detected by $c$ and has no indeterminacy.
\end{lem}

\begin{proof}
The proof follows several steps:
\begin{enumerate}
\item Establish the Massey product $c=\langle h_2, h_1, h_2\rangle$ in the $E_2$-page of the mANss.
\item Check that there are no crossing differentials.
\item Check that the Toda bracket $\langle \nu, \eta, \nu \rangle$ is well-defined and that it has no indeterminacy.
\item Apply the Moss convergence theorem to the Massey product and deduce the desired Toda bracket.
\end{enumerate}

For step (1), we check the following statements:
\begin{enumerate}[label=(\alph*)]
\item 
The Massey product is well-defined because
of the relation $h_1h_2=0$ in the $E_2$-page of the mANss for $\mmf$ (see \cref{fig:E2d7}).
\item The element $c$ is contained in the Massey product \cite[Equation (7.3)]{Bau08} \cite{Baer}.
\item The indeterminacy is trivial by inspection.
In more detail, the indeterminacy equals $h_2 \cdot E_2^{5,1,3}$.
The only non-zero element of $E_2^{5,1,3}$ is $h_1 v_1^2$,
and $h_2 \cdot h_1 v_1^2 = 0$.  This last relation 
holds already in the $E_2$-page of the motivic algebraic Novikov
spectral sequence \cite{Baer}.
\end{enumerate}

For step (2), we need to check for crossing differentials for the relation
$h_1 h_2$ in degree $(4,2,3)$.  We are looking for non-zero
Adams-Novikov differentials
in degrees $(5,f,3)$, where $f < 1$.  There are no possible sources for
such differentials (see \cref{fig:E2d7}).

For step (3), we check that the Toda bracket is well-defined 
because $\eta \nu$ is zero in $\pi_{4,3 }\mmf$ for degree reasons.
The indeterminacy equals $\nu \cdot \pi_{5,3}\mmf$, which is zero for degree reasons.

For step (4), we apply the Moss convergence theorem.  The theorem implies 
that there exists an element in $\langle h_2, h_1, h_2 \rangle$ that is a permanent cycle and that detects an element in $\langle \nu, \eta, \nu \rangle$. Since there are no indeterminacies for both the Massey product and the Toda bracket, the permanent cycle must be $c$.
\end{proof}

\section{Differentials}	\label{sec:diff}
In this section, we compute all differentials in the mANss for $\mmf$,
proving hidden extensions and Toda brackets only as needed along the way. 
Our results are presented in logical order, so each proof only depends on
earlier results.  We return to a more exhaustive study
of hidden extensions later in \cref{Sec:ext}.

\begin{thm}
\cref{diff} lists all of the non-zero differentials on
all of the indecomposable elements of each mANss $E_r$-page.
\end{thm}

\begin{proof}
The differentials are proved in the various propositions later in this section.
The last column of \cref{diff} indicates the specific proposition that proves
each differential.

Some indecomposables do not support differentials.  In most cases, this 
follows for degree reasons, i.e., because there are no possible targets.
\cref{d9-D^4c} handles two slightly more difficult cases.
\end{proof}

All differentials follow from straightforward
applications of the Leibniz rule to the ones listed in \cref{diff}.

\subsection{$d_3$ differentials}	\label{subsec:d3}

\begin{prop}
\label{prop:d3}
$(5, 1, 3)$
$d_3(h_1v_1^2)=\tau h_1^4$.
\end{prop}

\begin{proof}
In the mAss $E_2$-page, $h_1^4$ is a non-zero element that is annihilated by $\tau$.
By inspection, 
$h_1^4$ corresponds to the element of the same name in the mANss.
Therefore, $\tau h_1^4$ must be hit by an Adams-Novikov differential,
and there is only one possibility.
\end{proof}

\begin{prop}
\label{prop:d3bopattern}
$(12, 0, 6)$
$d_3(4a) = \ta P h_1^3$.
\end{prop}

\begin{proof}
For degree reasons, $d_3(P) = 0$. Thus \Cref{prop:d3} implies that
$d_3(P \cdot h_1v_1^2) = \ta P h_1^4.$ 
We have the relation $P \cdot h_1 v_1^2 = h_1 \cdot 4 a$ in the Adams-Novikov $E_2$-page. Note that this relation arises from a hidden
$h_1$ extension from $h_0^2 a$ to $P \overline{h_1^4}$ in the algebraic Novikov spectral sequence \cite{Baer}.
Therefore, $4a$ must also support a $d_3$ differential, and there 
is only one possibility.
\end{proof}

The Leibniz rule, combined with \cref{prop:d3} and \cref{prop:d3bopattern},
implies some additional $d_3$ differentials.
By inspection, the other multiplicative generators do not support
$d_3$ differentials. 

\begin{rem}
All of the $d_3$ differentials are $h_1$-periodic, in the sense that they
can be computed in the localization of the mANss $E_2$-page in which
$h_1$ is inverted.  This localized spectral sequence computes the
homotopy of the $\eta$-periodic spectrum $\mmf[\eta^{-1}]$.
See \cite[Section 6.1]{GI15} for a related discussion.
\end{rem}

\subsection{Corresponding pairs}
Earlier in \cref{sec:ASSANSS}, we discussed the notion of elements from the mANss
and from the mAss that correspond.  Having computed the $d_3$ differentials,
we are now in a position to establish a number of corresponding
pairs that will be used in later arguments.
	
\begin{thm}
\label{thm:ASSANSS}
\cref{table:ASSANSS} lists some pairs of elements that correspond.
\end{thm}

\begin{longtable}{llll}
\caption{Some corresponding elements in the motivic Adams and motivic Adams-Novikov spectral sequences
\label{table:ASSANSS} 
} \\
\toprule
mANss degree & mANss element & mAss element & mAss degree  \\
\midrule \endhead
$(0,0, 0)$ & $2$ & $h_0$ & $(0,1,0)$ \\
$(1,1,1)$ & $h_1$ & $h_1$ & $(1,1,1)$ \\
$(3,1, 2)$ & $h_2$ & $h_2$ & $(3,1,2)$ \\
$(14,2, 8)$ & $d$ & $d$ & $(14,4,8)$ \\
$(20,4,12)$&  $g$ & $g$ & $(20,4,12)$ \\
$(25,1,13)$&  $\Delta h_1$ & $\Delta h_1$ & $(25,5,13)$ \\
$(27,1,14)$ & $2 \Delta h_2$  & $an$ & $(27, 6, 14)$ \\
$(48,0,24)$ & $4 \Delta^2$ & $\Delta^2 h_0^2$ & $(48,10,24)$ \\
$(110,2,56)$ & $\Delta^4 d$ & $\Delta^4 d$ & $(110, 20, 56)$ \\
\bottomrule
\end{longtable}

\begin{proof}
We discuss the correspondence between $2 \Delta h_2$ and $a n$ in detail.
Most of the other corresponding pairs are established with essentially the same argument.
Some slightly more difficult cases are established later in 
Lemmas 
\ref{lem:4D^2-correspond} and \ref{D^4d-correspond}.

For degree reasons, the element $2 \Delta h_2$ of the mANss for $\mmf$ 
 cannot 
support an Adams-Novikov differential, nor can it be hit by an
Adams-Novikov differential.
(Beware that $\Delta h_2$ does support a differential.)
Therefore, $2 \Delta h_2$ detects some element $\alpha$ in $\pi_{27,14} \mmf$.

The inclusion $i: \mmf \rightarrow \mmf/\tau$ induces a map
\begin{equation}
\label{eq:inclusion-Ass}
\xymatrix{
E_2(\mmf) \ar[r] \ar@{=>}[d] & E_2(\mmf/\tau) \ar@{=>}[d] \\
\pi_{*,*} \mmf \ar[r] & \pi_{*,*} \mmf/\tau
}
\end{equation}
of motivic Adams spectral sequences.  
The spectral sequence on the right is identified with
the algebraic Novikov spectral sequence that converges to the
classical Adams-Novikov $E_2$-page for $\tmf$ \cite{BWX21}.

The element $\alpha$ in the lower left corner
maps to $2 \Delta h_2$ in the lower right corner.
This latter element is detected by $a n$ in filtration
$6$ in the upper right corner \cite{Baer}.
Therefore, $\alpha$ is detected in the upper left corner in filtration at most $6$.
The only possible value is $a n$.
\end{proof}

\begin{rem}
Previous knowledge of the $d_3$ differentials is required in order to
conclude that $2 \Delta h_2$ (and other elements as well) does not
support an Adams-Novikov differential.  For example,
it is conceivable that $d_{25} (2 \Delta h_2) = \tau^{12} h_1^{26}$.
However, we already know that $\tau^{12} h_1^{26}$ is hit by the differential
$d_3( \tau^{11} h_1^{22} \cdot h_1 v_1^2)$.
\end{rem}

\subsection{$d_5$ differentials}
\label{subsctn:d5}

Having determined all $d_3$ differentials, one can mechanically compute
the $E_4$-page.  Through the 22-stem, no additional differentials
are possible for degree reasons, so the $E_4$-page equals the $E_\infty$-page
in that range.

\begin{prop}
\label{prop:d5_D}
$(24, 0, 12)$
There exists a generator $\Delta$ of 
the mANss $E_2$-page in degree $(24,0,12)$ such that
$d_5(\Delta)=\tau^2 h_2 g.$
\end{prop}

\begin{proof}
The mAss element $h_2 g$ is annihilated by $\tau^2$ in the $E_2$-page.
Moreover, $\tau h_2 g$ does not support a hidden $\tau$ extension
in the mAss because of the presence of $\overline{\tau h_2 g}$ in
the homotopy of $\mmf/\tau$.  More precisely, projection to the top cell
takes $\overline{\tau h_2 g}$ to $\tau h_2 g$, so $\tau h_2 g$ must detect
homotopy elements that are annihilated by $\tau$.

The mANss element $h_2 g$ corresponds to the mAss element $h_2 g$
because of \cref{table:ASSANSS} and \cref{prop:correspond-product}.
Therefore, $\tau^2 h_2 g$ must be hit by an Adams-Novikov differential.
The only possibility is a $d_5$ differential whose source is in degree
$(24,0,12)$.  Since $\tau^2 h_2 g$ is not a multiple of $2$,
the source of the differential must be a generator.
\end{proof}

\begin{rem}
\label{Delta-defn}
$(24,0,12)$
\cref{prop:d5_D} does not uniquely specify $\Delta$.
Since $4 \tau^2 h_2 g$ is zero in the mANss $E_2$-page,
$\Delta$ is only well-defined up to multiples of $4$.
Later in \cref{nu-defn} we will make a further refinement in the
definition of $\Delta$.
Also note that the choice of $\Delta$ depends
on a previous choice of $h_2$, as in \cref{h2-choice}.
\end{rem}

The Leibniz rule, together with \cref{prop:d5_D}, implies additional
$d_5$ differentials.  The other multiplicative generators of the
$E_5$-page do not support differentials.

Of particular note is the differential 
\[
d_5(\Delta^2) = 2 \Delta d_5(\Delta) = 2 \tau^2 \Delta h_2 g.
\]
This easy computation is an Adams-Novikov version of Bruner's theorem
on the interaction between Adams differentials and
algebraic squaring operations \cite{BMM86} \cite{Bruner84}.
However, its corresponding Adams differential
$d_2(\Delta^2) = \tau^2 a n g$ is not as easy to obtain by direct
analysis of the Adams spectral sequence \cite{BR21}.
The difficulty is that $\Delta^2$ is not the value of an algebraic
squaring operation since $\Delta$ is not present in the Adams $E_2$-page.
By ``postponing" the differential that hits $\tau^2 h_2 g$ from
algebra to topology, we obtain an easier argument for the differential
on $\Delta^2$.

\begin{lem}
\label{lem:4D^2-correspond}
$(48, 0, 24)$
The element $4 \Delta^2$ of the mANss for $\mmf$ corresponds to
$\Delta^2 h_0^2$ in the mAss for $\mmf$.
\end{lem}

\begin{proof}
Having established that $d_5(\Delta^2) = 2 \tau^2 \Delta h_2 g$ 
as a consequence of the Leibniz rule and \cref{prop:d5_D},
we conclude that $4 \Delta^2$ does not support an Adams-Novikov
differential for degree reasons.  (Beware that
$2 \Delta^2$ does support a differential, but we do not need to know
that already.)
Note that $4 \Delta^2$ is detected in the algebraic Novikov spectral
sequence by $\Delta^2 h_0^2$, which has filtration $10$.
Using the argument in the proof of \cref{thm:ASSANSS}, we conclude that
$4 \Delta^2$ corresponds to an element in the mAss with filtration at most
$10$.  However, there are three possibilities: $\Delta^2$, $\Delta^2 h_0$,
and $\Delta^2 h_0^2$.

The top horizontal map of Diagram (\ref{eq:inclusion-Ass}) takes
$\Delta^2$ and $\Delta^2 h_0$ to elements of the same name.
These elements detect $\Delta^2$ and $2 \Delta^2$ in the
Adams-Novikov $E_2$-page.  This means that $4 \Delta^2$
cannot correspond to $\Delta^2$ or $\Delta^2 h_0$.
\end{proof}

\subsection{$d_7$ differentials}
\label{subsec:d7}

The main goal of this section is to establish some $d_7$ differentials in 
\cref{prop:d7} and \cref{prop:d7two}.  In order to obtain these
differentials, we will need some hidden extensions and some later differentials.
We establish these other results first, in order to preserve strict
logical order.
	
\begin{lem}\label{lemma:h0ext3}
$(3, 1, 2)$
There is a hidden $2$ extension from
$2h_2$ to $\ta h_1^3.$ 
\end{lem}
\begin{proof}
According to \cref{table:ASSANSS} and \cref{prop:correspond-product},
the mANss element $2 h_2$ corresponds to the mAss element $h_0 h_2$.
The element $h_0 h_2$ supports an $h_0$ extension in the mAss $E_2$-page, 
so $2 h_2$ must
support a $2$ extension in the mANss.  There is only one possible target
for this extension.
\end{proof}

\begin{rem}
\label{rem:h0ext3}
The hidden extension of \cref{lemma:h0ext3} is the first in an
infinite family of similar hidden extensions from the elements
$2 h_2 g^k$ to the elements $\tau h_1^3 g^k$.
For $k \geq 1$, these extensions are ``exotic'' in the sense that they
do not occur classically, since both $2 h_2 g^k$ and $h_1^3 g^k$ are the
targets of classical Adams-Novikov differentials.
\end{rem}

\begin{lem} \label{lem:2-Dh0h2}
$(27, 1, 14)$
There is a hidden $2$ extension from $2 \Delta h_2$
to $\tau \Delta h_1^3$.
\end{lem}

\begin{proof}
We already observed in \cref{table:ASSANSS} 
that $2 \Delta h_2$ and $\Delta h_1 \cdot h_1^2$
correspond to $a n$ and $\Delta h_1^3$ in the mAss.
In the mAss $E_2$-page, we have the relation $h_0 \cdot a n = \tau \Delta h_1^3$.
Therefore, there must be a hidden $2$ extension between the corresponding
Adams-Novikov elements.
\end{proof}

\begin{prop}
\label{prop:d7}
\mbox{}
\begin{enumerate}
\item 
$(24, 0, 12)$
$d_7(4\Delta)=\tau^3 h_1^3 g$.
\item
$(48, 0, 24)$
$d_7(2\Delta^2)=\tau^3 \Delta h_1^3 g$.
\end{enumerate}
\end{prop}

\begin{proof}
\cref{prop:d5_D} says that $\tau^2 h_2 g$ is hit by an Adams-Novikov differential, so $2 \tau^2 h_2 g$ is also hit by an Adams-Novikov differential.
\cref{rem:h0ext3} says that there is a hidden $2$ extension from
$2 h_2 g$ to $\tau h_1^3 g$.  Therefore, $\tau^3 h_1^3 g$ is hit
by a differential, and there is just one possible source for this differential.

The proof for the second differential is essentially the same.
We need a hidden $2$ extension from
$2 \Delta h_2 g$ to $\tau \Delta h_1^3 g$, which 
follows from \cref{lem:2-Dh0h2} and multiplication by $g$.
\end{proof}

\begin{rem}
\cref{prop:d5_D} and \cref{prop:d7} show that
both $2 \tau h_2 g^k$ and $\tau^2 h_1^3 g^k$ are annihilated
by $\tau$.  In hindsight, we can see that the hidden
$2$ extensions connecting them 
are examples of \cref{method:ctau}.
Their pre-images in $\mmf/\tau$
are $2 \Delta g^{k-1}$ and $4 \Delta g^{k-1}$, which are related by
$2$ extensions.

However, beware that we needed the hidden $2$ extension from
$2 h_2$ to $\tau h_1^3$ in order to establish the differential
$d_7(4 \Delta)$.  An independent proof of \cref{lemma:h0ext3} is necessary in order to avoid a circular argument.
\end{rem}

Before finishing the analysis of the $d_7$ differential in \cref{prop:d7two},
we deduce some higher differentials. 

\begin{prop} 
\label{prop:d13-2D^3h2}
$(75, 1, 38)$
$d_{13}(2\De^3 h_2)=\ta^6 dg^3$.
\end{prop}

\begin{proof}
We have the relation
$ang \cdot an = \ta^4 d g^3$
in the mAss $E_2$-page because of the relations
$a^2 n= \ta d \cdot \De h_1$ and $\De h_1 \cdot n = \ta^3 g^2$
\cite[Theorem 4.13]{Isa09}.
According to \cref{table:ASSANSS} and \cref{prop:correspond-product},
the mANss elements $2 \Delta h_2 g$, $2 \Delta h_2$, $d$, and $g$ correspond to the
mAss elements $a n g$, $a n$, $d$, and $g$.
This means that there is a hidden
$2 \Delta h_2$ extension from
$2 \Delta h_2 g$ to $\tau^4 d g^3$ in the mANss.

Using the Leibniz rule and \cref{prop:d5_D},
we already know that $2 \tau^2 \Delta h_2 g$ is hit by
the differential $d_5(\Delta^2)$.
Therefore, $\tau^6 d g^3$ must also be hit by a differential.
There are two possibilities for this differential:
$d_{11}(\tau \Delta^3 h_1^3)$ and $d_{13}(\Delta^3 h_2)$.
However, $\tau \Delta^3 h_1^3$ is a product
$\tau (\Delta h_1)^3$ of permanent cycles, so it cannot support a differential.
\end{proof}

\begin{rem}
The proof of \cref{prop:d13-2D^3h2} contains an example of \cref{method:ctau}.
There is a hidden $2 \Delta h_2$ extension from $2 \tau \Delta h_2 g$
to $\tau^5 d g^3$.  Both of these elements are annihilated by $\tau$.
Their pre-images under projection to the top cell of $\mmf/\tau$
are $\Delta^2$ and $2 \Delta^3 h_2$ respectively, which are related
by a $2 \Delta h_2$ extension.
\end{rem}

\begin{prop}
\label{prop:d9-D^2c}
$(56, 2, 29)$
$d_9(\Delta^2 c)=\ta^4 h_1 d g^2$.
\end{prop}

\begin{proof}
Recall from \cref{exmp:hidextn} that there is a hidden $2$ extension
from $\Delta^2 h_2^2$ to $\tau^4 d g^2$.
The argument for this hidden extension uses
\Cref{prop:d5_D} and \Cref{prop:d13-2D^3h2}.
Therefore, $\tau^4 h_1 d g^2$ must be hit by a differential
because $2 h_1 = 0$.  There is only one possible differential.
\end{proof}

\begin{prop}
\label{Adams-diff}
In the mAss for $\mmf$, we have the Adams differentials:
\begin{enumerate}
\item
$(48, 8, 24)$
$d_2(\Delta^2) = \tau^2 a n g$.
\item
$(96, 16, 48)$
$d_3(\Delta^4) = \tau^8 n g^4$.
\end{enumerate}
\end{prop}

\begin{proof}
We start with the Adams-Novikov differential 
$d_5(\Delta^2) = 2 \tau^2 \Delta h_2 g$.  We know from \cref{table:ASSANSS}
and \cref{prop:correspond-product} that $2 \Delta h_2 g$ corresponds to 
the element $a n g$ in the mAss.  Therefore, $\tau^2 a n g$ must be hit
by some Adams differential, and the only possibility is that
$d_2(\Delta^2)$ equals $\tau^2 a n g$.

Next, we apply Bruner's theorem on the interaction between
Adams differentials and algebraic squaring operations.
We refer to \cite[Theorem 5.6]{BR21} for a precise readable statement,
although \cite{Bruner84}, \cite{BMM86} and \cite{Makinen73} are preceding references.
We apply Bruner's theorem with
$x = \Delta^2$, $r = 2$, and $i = 8$;
so $s = 8$, $t = 56$, $v = v(48) = 1$, and $\overline{a} = h_0$.
We obtain that
\[
d_* \Sq^8 (\Delta^2) = \Sq^9 d_2(\Delta^2) \dotplus h_0 \cdot \Delta^2 \cdot d_2(\Delta^2)=
\Sq^9(\tau^2 a n g) + h_0 \cdot \Delta^2 \cdot \tau^2 a n g =
\Sq^9(\tau^2 a n g).
\]
Next, we compute that 
$\Sq^9(\tau^2 a n g) = 
\tau^4 \cdot \tau \Delta h_1 \cdot n^2 \cdot g^2$, using the
Cartan formula for algebraic squaring operations, as well as the 
formulas $\Sq^2(a) = \tau \Delta h_1$, $\Sq^3(n) = n^2$, and
$\Sq^4(g) = g^2$ \cite[Theorem 1.20]{BR21}.
Finally, apply the relation $\Delta h_1 \cdot n = \tau^3 g^2$ to obtain
the Adams differential $d_3(\Delta^4) = \tau^8 n g^4$.
\end{proof}

\begin{rem}
\label{rem:Bruners-thm}
The careful reader may object to our use of a motivic version of
Bruner's theorem in the proof of \cref{Adams-diff}, while only the classical version of the theorem
has a published proof.  In fact, this concern is irrelevant here.
One can use the classical Bruner's theorem to establish the classical
Adams $d_3$ differential and then deduce the motivic version of the differential.
\end{rem}

\begin{prop}\label{prop:d7two}
$(96, 0, 48)$
$d_7(\Delta^4)=\ta^3 \De^3 h_1^3 g$.
\end{prop}

\begin{proof}
\cref{table:ASSANSS} shows that 
the mANss element $4 \Delta^2$ corresponds to the mAss element $\Delta^2 h_0^2$.
Therefore, \cref{prop:correspond-product} shows that 
the mANss element $16 \Delta^4$ corresponds to the mAss 
element $\Delta^4 h_0^4$.

\cref{Adams-diff} shows that $\Delta^4$ does not survive the mAss.
Therefore, $\Delta^4 h_0^4$ does not detect homotopy elements that are divisible
by $16$.  Consequently, the corresponding element $16 \Delta^4$ in the mANss
does not detect homotopy elements that are divisible by $16$.
This means that $\Delta^4$ must support an Adams-Novikov differential.

There are two possible values for this differential:
$\tau^3 \Delta^3 h_1^3 g$ and $\tau^9 h_1 d g^4$.
However, 
\cref{prop:d9-D^2c} shows that
the latter element is already hit by the differential
$d_9(\tau^5 \Delta^2 c g^2) = \tau^9 h_1 d g^4$.
\end{proof}

\subsection{$d_9$ differentials}
\label{subsec:d9}

At this point, we have determined all differentials $d_r$ for $r \leq 7$.
It remains to study higher differentials, although some higher differentials
have already been determined in earlier propositions. 
We continue to proceed roughly in order of increasing 
values of $r$, although we 
occasionally 
need some Toda brackets, hidden extensions, and later differentials
to preserve strict logical order.

\begin{prop}
\label{prop:d13-2D^7h2}
$(171, 1, 86)$
$d_{13}(2\De^7 h_2)=\ta^6 \De^4 dg^3$.
\end{prop}

\begin{proof}
The argument is nearly identical to the proof of \cref{prop:d13-2D^3h2}.
The mAss $E_2$-page relation 
$\Delta^4 ang \cdot an = \ta^4 \Delta^4 d g^3$
implies that there is a hidden
$2 \Delta h_2$ extension from
$2 \Delta^5 h_2 g$ to $\tau^4 \Delta^4 d g^3$ in the mANss.
We already know that $2 \tau^2 \Delta^5 h_2 g$ is hit by
the differential $d_5(\Delta^6)$.
Therefore, $\tau^6 \Delta^4 d g^3$ must also be hit by a differential.

There are two possibilities for this differential:
$d_{11}(\tau \Delta^7 h_1^3)$ and $d_{13}(2\Delta^7 h_2)$.
The former possibility is ruled out by the decomposition
$\tau \Delta^6 h_1^2 \cdot \Delta h_1$ and the observation that both
$\Delta^6 h_1^2$ and $\Delta h_1$ survive past the $E_{11}$-page for
degree reasons.
\end{proof}

\begin{lem}
\label{2-D^6h2^2}
$(150, 2, 76)$
There is a hidden $2$ extension from $\Delta^6 h_2^2$ to $\tau^4 \Delta^4 d g^2$.
\end{lem}

\begin{proof}
The proof is similar to the argument in \cref{exmp:hidextn}.
We already know the differentials
$d_5(\Delta^7 h_2) = \tau^2 \Delta^6 h_2^2 g$ and
$d_{13}(2\De^7 h_2)=\ta^6 \De^4 dg^3$ from
Propositions \ref{prop:d5_D} and \ref{prop:d13-2D^7h2}.
Therefore, projection to the top cell detects a hidden $2$ extension
from $\tau \Delta^6 h_2^2 g$ to $\tau^5 \Delta^4 d g^3$.
Finally, use $\tau g$ multiplication to deduce the hidden $2$ 
extension on $\Delta^6 h_2^2$.
\end{proof}

\begin{prop}
\label{prop:d9one}
\mbox{}
\begin{enumerate}
\item
$(80, 2, 41)$
$d_9(\De^3 c)=\ta^4 \De h_1 d g^2$.
\item
$(176, 2, 89)$
$d_9(\De^7 c)=\ta^4 \De^5 h_1 d g^2$.
\end{enumerate}
\end{prop}

\begin{proof}
We saw in \cref{exmp:hidextn} that $\tau^4 d g^2$ detects a multiple of $2$.
Therefore, $\Delta h_1 \cdot \tau^4 d g^2$ must detect zero since
$\Delta h_1$ does not support a $2$ extension for degree reasons.
Therefore, $\tau^4 \Delta h_1 d g^2$ must be hit by a differential,
and there is only one possibility.

The argument for the second differential is nearly identical.
\cref{2-D^6h2^2} shows that the element
$\tau^4 \Delta^4 d g^2$ detects a multiple of $2$.
Therefore, $\Delta h_1 \cdot \tau^4 \Delta^4 d g^2$ must detect zero,
and there is only one differential that can hit it.
\end{proof}

\begin{prop}
\label{prop:d9-D^6c}
$(152, 2, 77)$
$d_9(\De^6 c)=\ta^4 \De^4 h_1 d g^2$.
\end{prop}

\begin{proof}
The argument is similar to the proof of \cref{prop:d9-D^2c}.
\cref{2-D^6h2^2} shows that $\tau^4 \Delta^4 d g^2$ detects a multiple
of $2$.  Therefore, $\tau^4 \Delta^4 h_1 d g^2$ must be hit by a differential
because $2 h_1 = 0$.
There is only one possible differential.
\end{proof}

\begin{lem}\label{lemma:toda}
$(25, 1, 13)$
The Toda bracket 
$\langle\eta , \nu, \ta^2 \bar \kappa\rangle$ 
is detected by $\Delta h_1$ and has indeterminacy detected by
$P^3 h_1$.
\end{lem}

\begin{proof}
By inspection, the Toda bracket is well-defined and has indeterminacy 
detected by $P^3 h_1$ (which is a $v_1$-periodic element).

We use the Moss convergence theorem in the mAss for $\mmf$.
By \cite[Definition 4.4(1)]{Isa09}, we have the Massey product 
$\Delta h_1=\langle h_1, h_2, \ta^2 g\rangle$ 
in the $E_2$-page of the mAss for $\mmf$.
There are no possible crossing differentials in the mAss for $\mmf$.

Finally, \cref{table:ASSANSS} implies that the mAss elements
$h_1$, $h_2$, and $\tau^2 g$ detect $\eta$, $\nu$, and $\tau^2 \kappabar$
respectively (see also \cref{htpy-name}).
\end{proof}

\begin{lem}
\label{lemma:h2ext25}
$(25, 1, 13)$
There is a hidden $\nu$ extension from $\Delta h_1$ to $\ta^2 c g$.
\end{lem}

\begin{proof}
Lemmas \ref{lemma:todaepsilon} and \ref{lemma:toda} show that
the Toda brackets
$\langle \nu,\eta, \nu\rangle$ and 
$\langle \eta, \nu, \ta^2 \bar\kappa \rangle$ are detected by
$c$ and $\Delta h_1$ respectively.

The hidden $\nu$ extension follows from the shuffling relation	
$$\nu \langle \eta, \nu, \ta^2 \bar\kappa \rangle = 
\langle \nu, \eta, \nu \rangle \ta^2 \bar\kappa.$$
\end{proof}
	
\begin{lem}
\label{lemma:h1ext27}
$(25, 1, 13)$
There is a hidden $\eta$ extension from $2\Delta h_2 $ to $\ta^2 c g$.
\end{lem}

\begin{proof}
As in the proof of \cref{lemma:h2ext25}, the element $\tau^2 c g$
detects $\langle \eta, \nu, \tau^2 \kappabar \rangle \nu$,
which equals $\eta \langle \nu, \tau^2 \kappabar, \nu \rangle$.
Therefore, $\tau^2 c g$ is the target of a hidden $\eta$ extension.
There are two possible sources for such an extension:
$\tau \Delta h_1^3$ and $2 \Delta h_2$.
The former possibility is ruled out by 
\cref{lem:2-Dh0h2}, which shows that $\tau \Delta h_1^3$
is the target of a hidden $2$ extension.
\end{proof}

\begin{prop}
\label{prop:d9two}
\mbox{}
\begin{enumerate}
\item 
$(49, 1, 25)$
$d_9(\Delta^2 h_1)=\tau^4 c g^2$.
\item
$(73, 1, 37)$
$d_9(\Delta^3 h_1)=\tau^4 \Delta c g^2$.
\item
$(145, 1, 73)$
$d_9(\Delta^6 h_1)=\tau^4 \Delta^4 c g^2$.
\item
$(169, 1, 85)$
$d_9(\Delta^7 h_1)=\tau^4 \Delta^5 c g^2$.
\end{enumerate}
\end{prop}

\begin{proof}
It follows from \cref{lemma:h1ext27} that 
there is a hidden $\eta$ extension from $2 \Delta h_2 g$ to
$\tau^2 c g^2$.
\cref{prop:d5_D} and the Leibniz rule imply that
$d_5(\Delta^2) = 2 \tau^2 \Delta h_2 g$.
Therefore,
$\tau^4 c g^2$ must be hit by some differential,
and there is only one possibility.

Having established the first differential, we can compute that
\[
d_9(\Delta^3 h_1^2) = \Delta h_1 \cdot d_9(\Delta^2 h_1) =
\tau^4 \Delta h_1 c g^2.
\]
Since $\Delta^3 h_1^2 = \Delta^3 h_1 \cdot h_1$, it follows that
$d_9(\Delta^3 h_1)$ equals $\tau^4 \Delta c g^2$.

Similarly,
\[
d_9(\Delta^7 h_1^2) = \Delta^5 h_1 \cdot d_9(\Delta^2 h_1) =
\tau^4 \Delta^5 h_1 c g^2,
\]  
from which it follows that
$d_9(\Delta^7 h_1)$ equals $\tau^4 \Delta^5 c g^2$.
However, we need to observe that $d_9(\Delta^5 h_1)$ is zero.
The only possible non-zero value for $d_9(\Delta^5 h_1)$
is $\tau^4 \Delta^3 cg^2$, but this is ruled out by the observation
that $\tau^4 \Delta^3 cg^2$ supports a $d_9$ differential by \cref{prop:d9one}.

Finally, note that
$d_9(\Delta^7 h_1^2) = \Delta h_1 \cdot d_9(\Delta^6 h_1)$.
The value of $d_9(\Delta^7 h_1^2)$ was computed in the previous paragraph.
It follows that $d_9(\Delta^6 h_1)$ equals $\tau^4 \Delta^4 c g^2$.
\end{proof}

\begin{prop}
\label{d9-D^4c}
\mbox{}
\begin{enumerate}
\item
$d_9(\Delta^4 c) = 0$.
\item
$d_9(\Delta^5 c) = 0$.
\end{enumerate}
\end{prop}

\begin{proof}
It follows from \cref{prop:d9two} that
$\tau^4 \Delta^4 c g^2$ and $\tau^4 \Delta^5 c g^2$ are targets of
$d_9$ differentials, so they cannot support $d_9$ differentials.
This implies that $\Delta^4 c$ and $\Delta^5 c$ cannot support
$d_9$ differentials.
\end{proof}

The Leibniz rule, together with the differentials given
in Propositions \ref{prop:d9one}, \ref{prop:d9-D^6c}, \ref{prop:d9two},
and \ref{d9-D^4c},
determines all $d_9$ differentials.

\subsection{$d_{11}$ differentials}
\label{subsec:d11}

\begin{lem}\label{lemma:ext-cd}
$(14, 2, 8)$
There is a hidden $\epsilon$ extension from $d$ to $\ta h_1^2 g$.
\end{lem}

\begin{proof}
We will show that there is a hidden $\epsilon$ extension
from $h_1 d$ to $\tau h_1^3 g$.  The desired extension follows 
immediately.

The relation $h_1 c = h_2^3$ in the mANss $E_2$-page implies that
$\eta \epsilon$ equals $\nu^3$.  
Also, the relation $h_2^2 d = 4 g$ implies that $\nu^2 \kappa = 4 \kappabar$.
Then
\[
\eta \epsilon \kappa = \nu^3 \kappa = 4 \nu \kappabar =
\tau \eta^3 \kappabar.
\]
The last equality uses the hidden $2$ extension
from $2 h_2$ to $\tau h_1^3$, as shown in \cref{lemma:h0ext3}.
\end{proof}

\begin{lem}
\label{lemma:h2ext39}
$(39, 3, 21)$
There is a hidden $\nu$ extension from $\Delta h_1 d $ to $\ta^3 h_1^2 g^2.$
\end{lem}

\begin{proof}
The element $\Delta h_1 d$ detects the product
$\eta_1 \cdot \kappa$.
\cref{lemma:h2ext25} implies that $\nu \cdot \eta_1 \cdot \kappa$ equals
$\tau^2 \epsilon \kappa \kappabar$. 
\cref{lemma:ext-cd} implies that this last product equals
$\tau^3 \eta^2 \kappabar^2$, which is detected by $\tau^3 h_1^2 g^2$.
\end{proof}

\begin{prop}
\label{prop:d11}
\mbox{}
\begin{enumerate}
\item 
$(62, 2, 32)$
$d_{11}(\De^2 d)=\ta^5 h_1 g^3$.
\item 
$(158, 2, 80)$
$d_{11}(\De^6 d)=\ta^5 \De^4 h_1 g^3$.
\end{enumerate}
\end{prop}

\begin{proof}
The element $\tau^5 h_1^2 g^3$ detects $\tau^5 \eta^2 \kappabar^3$.
\cref{lemma:h2ext39} implies that 
$\tau^5 \eta^2 \kappabar^3$ equals 
$\tau^2 \nu \kappabar \cdot \eta_1 \cdot \kappa$.
Because of \cref{prop:d5_D}, we know that $\tau^2 \nu \kappabar$ is zero.
Therefore, $\tau^5 h_1^2 g^3$ is hit by some differential.
The only possibility is that $d_{11}(\Delta^2 h_1 d) = \tau^5 h_1^2 g^3$.
It follows that $d_{11}(\Delta^2 d) = \tau^5 h_1 g^3$.

For the second formula, multiply by the permanent cycle $\Delta^4 h_1$ to 
see that
$d_{11}(\Delta^6 h_1 d)$ equals $\tau^5 \Delta^4 h_1^2 g^3$.
It follows that 
$d_{11}(\Delta^6 d)$ equals $\tau^5 \Delta^4 h_1 g^3$.
\end{proof}

\subsection{$d_{13}$ differentials}
\label{subsec:d13}

We have already established some $d_{13}$ differentials in
Propo\-si\-tions \ref{prop:d13-2D^3h2} and \ref{prop:d13-2D^7h2} because we needed 
those results in order to compute shorter differentials.
We now finish the computation of the $d_{13}$ differentials.

\begin{lem}
\label{D^4d-correspond}
$(110, 2, 56)$
The element $\Delta^4 d$ of the mANss for $\mmf$ corresponds to the
element of the same name in the mAss for $\mmf$.
\end{lem}

\begin{proof}
We have already analyzed all possible 
Adams-Novikov differentials of length $11$ or less, 
and there are no other possible
values for a differential on $\Delta^4 d$.
Therefore, $\Delta^4 d$ is a permanent cycle in the
mANss for $\mmf$.

Now the argument given in the proof of \cref{thm:ASSANSS} applies.
The mANss element $\Delta^4 d$ is detected in filtration $20$ in the
Adams $E_2$-page for $\mmf/\tau$.  Therefore, $\Delta^4 d$ corresponds
to an element of the mAss with Adams filtration at most $20$.
There is only one possible element in the mAss with sufficiently low filtration.
\end{proof}
\begin{lem}
\label{lemma:h1ext39}
\mbox{}
\begin{enumerate}
\item
$(39, 3, 21)$
There is a hidden $\eta$ extension from $\Delta h_1 d$ to $2 \ta^2 g^2$.
\item
$(135, 3, 69)$
There is a hidden $\eta$ extension from $\Delta^5 h_1 d$ to 
$2 \ta^2 \Delta^4 g^2$.
\end{enumerate}
\end{lem}

\begin{proof}
\cref{table:ASSANSS} shows that 
the elements $\Delta h_1$ and $d$ in the mANss for $\mmf$ correspond
to elements of the same name in the mAss for $\mmf$.
The product $\Delta h_1 \cdot h_1 d$ is non-zero in the mAss $E_2$-page
and also in the mAss $E_\infty$-page because there are no possible
differentials that could hit it. (Note that this product is non-zero
in the motivic context, but the corresponding classical product is zero
in the $E_2$-page of the Adams spectral sequence for $\tmf$.)

Therefore, $\Delta h_1 d$ must support a hidden $\eta$ extension in
the mANss for $\mmf$.  There are three possible targets for this extension:
$\tau^2 g^2$, $2 \tau^2 g^2$, and $3 \tau^2 g^2$.  The first and last possibilities are ruled out by the relation $2 \eta = 0$.

The argument for the second extension is nearly identical.
\cref{table:ASSANSS} and \cref{prop:correspond-product} imply
that the mANss element $\Delta^5 h_1 d$ corresponds to the
mAss element $\Delta^4 \cdot \Delta h_1 \cdot d$.
The product $\Delta^4 \cdot \Delta h_1 \cdot h_1 d$ 
is non-zero in the mAss $E_\infty$-page, so
$\Delta^5 h_1 d$ must support a hidden $\eta$ extension in the
mANss.  The only possible target for this extension
is $2 \tau^2 \Delta^4 g^2$.
\end{proof}
	
\begin{prop}
\label{prop:d13two}
\mbox{}
\begin{enumerate}
\item 
$(81, 3, 42)$
$d_{13}(\Delta^3 h_1 c) =  2 \tau^6 g^4$.
\item
$(177, 3, 90)$
$d_{13}(\Delta^7 h_1 c) =  2 \tau^6  \Delta^4 g^4$.
\end{enumerate}
\end{prop}

\begin{proof}
\cref{lemma:h1ext39} implies that there is a hidden $\eta$ extension 
from $\Delta h_1 d g^2$ to $2 \tau^2 g^4$.
\cref{prop:d9one} shows that $\tau^4 \Delta h_1 d g^2$ is hit by a differential.
Therefore, $2 \tau^6 g^4$ must also be hit by a differential.
There is only one possible source for this differential.

The proof for the second formula is similar.
There is a hidden $\eta$ extension from
$\Delta^5 h_1 d g^2$ to $2 \tau^2 \Delta^4 g^4$.
Since $\tau^4 \Delta^5 h_1 d g^2$ is hit by a differential,
$2 \tau^6 \Delta^4 g^4$ must also be hit by a differential.
\end{proof}

\subsection{$d_{23}$ differentials}
\label{subsec:d23}

\begin{lem}
\label{Dh1-tD^3h1^3}
$(75, 3, 38)$
There is a hidden $\eta_1$ extension from
$\tau \Delta^3 h_1^3$ to $\tau^9 g^5$.
\end{lem}

\begin{proof}
According to \cref{table:ASSANSS}, the mANss elements $\Delta h_1$ and
$g$ correspond to elements of the same name in the mAss.
In the mAss $E_2$-page, the relations given in \cite[Theorem 4.13]{Isa09}
imply that $\tau (\Delta h_1)^4 = \tau^9 g^5$.
Therefore, in the mANss, $\tau^9 g^5$ detects the product 
$\tau \eta_1^4$.
On the other hand, $\tau \Delta^3 h_1^3$ detects the product
$\tau \eta_1^3$ in the mANss. 
\end{proof}

\begin{rem}
$(75,3,39)$
Beware that $\Delta^3 h_1^3$ does not support a hidden $\eta_1$ extension.
Rather, it supports a non-hidden extension since $\Delta^4 h_1^4$ is non-zero.
However, $\Delta^4 h_1^4$ is annihilated by $\tau$, which allows for
the hidden extension on $\tau \Delta^3 h_1^3$.
\end{rem}

\begin{prop} 
\label{prop:d23-D^5h1}
$(121, 1, 61)$
$d_{23}(\De^5 h_1)=\ta^{11}g^6$.
\end{prop}

\begin{proof}
The hidden extension of \cref{Dh1-tD^3h1^3} implies that
there is a hidden $\eta_1$ extension from
$\tau \Delta^3 h_1^3 g$ to $\tau^9 g^6$.
We already know that $\tau^3 \Delta^3 h_1^3 g$ is zero because
of the differential $d_7(\Delta^4)$ from \cref{prop:d7two}.
Therefore, $\tau^{11} g^6$ must be the value of some differential,
and there is only one possibility.
\end{proof}

\section{Hidden extensions}
\label{Sec:ext}

In \cref{sec:diff}, we established several hidden extensions in the
mANss for $\mmf$ as steps towards computing differentials.
In this section, we finish the analysis of all hidden extensions
by $2$, $\eta$, and $\nu$.
Our work does not completely determine the ring structure of
$\pi_{*,*} \mmf$ because there exist hidden extensions by other elements.
Up to one minor uncertainty, the entire ring structure 
of $\pi_* \tmf$ is determined in \cite{BR21}.

\begin{thm}
Up to multiples of $g$ and $\Delta^8$,
Tables \ref{hidh0}, \ref{hidh1} and \ref{hidh2} list all 
hidden extensions by $2$, $\eta$, and $\nu$ in the 
mANss for $\mmf$.
\end{thm}

\begin{proof}
Some of the non-zero hidden extensions are established in the previous
results because we needed them to compute Adams-Novikov differentials.
The remaining non-zero hidden extensions are proved in the
following results.
The last columns of the tables indicate the specific
proofs for each extension.

There are some possible hidden extensions
that turn out not to occur.  Most of these possibilities can be
ruled out using \cref{method:ctau}.  For example, consider the
possible hidden $\eta$ extension from $\tau \Delta h_1^3$
to $\tau^2 c g$.  Because of multiplication by $\tau g$,
we may instead consider the possible hidden $\eta$ extension
from $\tau^2 \Delta h_1^3 g$ to $\tau^3 c g^2$.
These last two elements are annihilated by $\tau$, so they
are in the image of projection to the top cell.
By inspection, there is no $\eta$ extension in the homotopy
of $\mmf/\tau$ in the appropriate degree.

A few miscellaneous cases remain, but their proofs are straightforward.  For example,
\begin{itemize}
\item
$(65,3,34)$ there is no hidden $2$ extension from 
$\Delta^2 h_2 d$ to $\tau^3 \Delta h_1 g^2$ because
the latter element supports an $h_1$ extension.
\item
$(24, 0, 12)$ there is no hidden $\nu$ extension from
$8 \Delta$ to $\tau \Delta h_1^3$ because the first 
element is annihilated by $g$ while the second element is not.
\end{itemize}
\end{proof}

\begin{prop}
\label{prop:he-mmf/t}
\Cref{hidtaumethod} lists some hidden extensions in the 
mANss for $\mmf$.
\end{prop}

\begin{longtable}{llllll}
\caption{Some hidden extensions deduced from \cref{method:ctau}
\label{hidtaumethod} 
} \\
\toprule
$(s,f,w)$ & source & type & target & reason  \\
\midrule \endfirsthead
\caption[]{Some hidden extensions deduced from \cref{method:ctau}} \\
\toprule
$(s,f,w)$ & source& type & target & reason \\
\midrule \endhead
\bottomrule \endfoot
$(51,1,26)$ & $2\Delta^2 h_2$ & $2$ &  $\tau \Delta^2 h_1^3$ &
$d_5(2\Delta^3)=2\tau^2 \Delta^2 h_2 g$ & $d_{7}(4\Delta^3 )=\tau^3 \Delta^2 h_1^3 g$\\
$(54,2,28)$ & $\Delta^2 h_2^2$ & $2$ & $\tau^4 d g^2$ & $d_5(\Delta^3 h_2) = \tau^2 \Delta^2 h_2^2 g$ & $d_{13}(2\Delta^3 h_2) = \tau^6 d g^3$ \\
$(99,1,50)$ &  $2\Delta^4 h_2$ & $2$ &  $\tau \Delta^4 h_1^3$ &
$d_5(2\Delta^5)=2\tau^2 \Delta^4 h_2 g$ & $d_{7}(4\Delta^5 )=\tau^3 \Delta^4 h_1^3 g$\\
$(123,1,62)$ & $2 \Delta^5 h_2$ & $2$ & $\tau \Delta^5 h_1^3$ & $d_5(\Delta^6) = 2 \tau^2 \Delta^5 h_2 g$ & $d_7(2 \Delta^6) = \tau^3 \Delta^5 h_1^3 g$ \\
$(147,1,74)$ &   $2\Delta^6 h_2$ & $2$ &   $\tau \Delta^6 h_1^3$ &
$d_5(2\Delta^7)=2\tau^2 \Delta^6 h_2 g$ & $d_{7}(4\Delta^7 )=\tau^3 \Delta^6 h_1^3 g$\\
$(51,1,26)$ & $\Delta^2 h_2$ & $\eta$ & $\tau^2\Delta c g$ &
$d_5(\Delta^3)=\tau^2 \Delta^2 h_2 g$ & $d_9(\Delta^3 h_1)=\tau^4 \Delta c g^2$ \\
$(99,1,50)$ & $\Delta^4 h_2$ & $\eta$ & $\tau^9 g^5$ & $d_5(\Delta^5) = \tau^2 \Delta^4 h_2 g$ & $d_{23}(\Delta^5 h_1) = \tau^{11} g^6$ \\
$(123,1,62)$ & $2\Delta^5 h_2$ & $\eta$ & $\tau^2 \Delta^4 c g$ &
$d_5(\Delta^6)=2 \tau^2 \Delta^5 h_2 g$ & $d_9(\Delta^6 h_1)=\tau^4 \Delta^4 c g^2$ \\
$(124, 6, 63)$ & $\tau^2 \Delta^4 c g$ & $\eta$ & $\tau^9 \Delta h_1 g^5$ &
$d_9(\Delta^6 h_1)=\tau^4 \Delta^4 c g^2$ & $d_{23}(\Delta^6 h_1^2)=\tau^{11} \Delta h_1 g^6$ \\
$(129,3,66)$ & $\Delta^5 h_1 c$ & $\eta$ & $\tau^7\Delta^2 h_1^2 g^4$ &
$d_9(\Delta^7 h_1^2)\mathord{=}\tau^4 \Delta^5 h_1 c g^2$ & $d_{23}(\Delta^7h_1^3)=\tau^{11}\Delta^2 h_1^2 g^6$ \\
$(147,1,74)$ & $\Delta^6 h_2$ & $\eta$ & $\tau^2 \Delta^5 c g$ &
$d_5(\Delta^7)=\tau^2 \Delta^6 h_2 g$ & $d_9(\Delta^7 h_1)=\tau^4 \Delta^5 c g^2$\\
$(161,3,82)$ & $\Delta^6 h_2 d$ & $\eta$ & $\tau^3 \Delta^5 h_1^2 g^2$ &
$d_5(\Delta^7 d)=\tau^2 \Delta^6 h_2 d g$ & $d_{11}(\Delta^7 h_1 d)\mathord{=}\tau^5 \Delta^5 h_1^2 g^3$\\
$(0,0,0)$ & $4$ & $\nu$ &  $\tau h_1^3$ & $d_5(\Delta h_2 d) = 4 \tau^2 g^2$ &
$d_7(4 \Delta g) = \tau^3 h_1^3 g^2$ \\
$(48,0,24)$ & $4\Delta^2 $ & $\nu$ &  $\tau \Delta^2 h_1^3$ & $d_5(\Delta^3 h_2 d) = 4 \tau^2 \Delta^2 g^2$ & $d_7(4 \Delta^3 g) = \tau^3 \Delta^2 h_1^3 g^2$ \\
$(51,1,26)$ &$2\Delta^2 h_2 $& $\nu$ &  $\tau^4 d g^2$ &
$d_5(2\Delta^3 )=2\tau^2 \Delta^2 h_2 g$ & $d_{13}(2\Delta^3 h_2)= \tau^6 d g^3$\\
$(57,3,30)$ & $\Delta^2 h_2^3$ & $\nu$ & $2 \tau^4 g^3$ &
$d_5(\Delta^3 h_2^2)=\tau^2 \Delta^2 h_2^3 g$ & $d_{13}(\Delta^3 h_2^3)=2 \tau^6 g^4$\\
$(96,0,48)$ &  $4\Delta^4$ & $\nu$ &  $\tau \Delta^4 h_1^3$ & $d_5(\Delta^5 h_2 d) = 4 \tau^2 \Delta^4 g^2$ & $d_7(4 \Delta^5 g) = \tau^3 \Delta^4 h_1^3 g^2$ \\
$(144,0,72)$ &  $4\Delta^6$ & $\nu$ &  $\tau \Delta^6 h_1^3$ & $d_5(\Delta^7 h_2 d) = 4 \tau^2 \Delta^6 g^2$ & $d_7(4 \Delta^7 g) = \tau^3 \Delta^6 h_1^3 g^2$ \\
$(147,1,74)$ &$2\Delta^6 h_2$ & $\nu$ &  $\tau^4 \Delta^4 d g^2$ &
$d_5(2\Delta^7)=2\tau^2 \Delta^6 h_2 g$ & $d_{13}(2\Delta^7 h_2)=\tau^6  \Delta^4 d g^3$\\
$(153,3,78)$ & $\Delta^6 h_2^3$ & $\nu$ & $2 \tau^4 \Delta^4 g^3$ &
$d_5(\Delta^7 h_2^2)=\tau^2 \Delta^6 h_2^3 g$ & $d_{13}(\Delta^7 h_2^3) = 2 \tau^6 \Delta^4 g^4$\\
\end{longtable}

\begin{proof}
All of these extensions follow from \cref{method:ctau}, using
the differentials in the last two columns of \Cref{hidtaumethod}.
To illustrate, we discuss the first extension in the table.
In order to obtain the extension from $2 \Delta^2 h_2$ to $\tau \Delta^2 h_1^3$,
we can establish a hidden $2$ extension from $2 \tau \Delta^2 h_2 g$ to 
$\tau^2 \Delta^2 h_1^3 g$.
Then the desired extension follows immediately.

The elements $2 \tau \Delta^2 h_2 g$ and $\tau^2 \Delta^2 h_1^3 g$ are annihilated by $\tau$ in the $E_\infty$-page of the mANss for $\mmf$.
Therefore, they detect elements in $\pi_{71,37} \mmf$ that are in the image
of $\pi_{72,36} \mmf/\tau$ under projection to the top cell.
By inspection, these preimages are $2 \Delta^3$ and $4 \Delta^3$.
These latter elements are connected by a $2$ extension, so their images
are also connected by a $2$ extension.

The other extensions have essentially the same proof.
First multiply by an appropriate power of $g$.  
Then pull back to $\pi_{*,*} \mmf/\tau$,
where the extension is visible by inspection.
\end{proof}

\begin{rem}
\label{eta-t^2D^4cg}
$(124,6,63)$
The hidden $\eta$ extension from $\tau^2 \Delta^4 c g$ to
$\tau^9 \Delta h_1 g^5$ in \cref{hidtaumethod} deserves further 
discussion.
Note that $\Delta^4 c g$ and $\tau \Delta^4 c g$ support $\eta$
extensions that are not hidden.  However,
$\tau^2 \Delta^4 h_1 c g$ is zero, so $\tau^2 \Delta^4 c g$
can support a hidden $\eta$ extension.
This explains why the $E_\infty$-page chart in \cref{fig:E-infty}
shows both an $h_1$ extension and a hidden $\eta$ extension on
the element $\Delta^4 c g$ in the 124-stem.

The subtleties of this situation are illuminated by consideration 
of homotopy elements.
Let $\alpha$ be an element of $\pi_{124,65} \mmf$ that is detected by
$\Delta^4 c g$.  
The element $\tau^2 \alpha$ is detected by $\tau^2 \Delta^4 c g$.
The hidden $\eta$ extension implies that
$\tau^2 \eta \alpha$ is detected by $\tau^9 \Delta h_1 g^5$.

Now let $\beta$ be an element in $\pi_{122,64}$ that is detected
by $\Delta^4 h_2^2 g$.  Note that $\tau^2 \beta$ must be zero because 
$\tau^2 \Delta^2 h_2^2 g$ is zero and because there are no $E_\infty$-page
elements in higher filtration.
Then $\nu \beta$ is detected by $h_2 \cdot \Delta^4 h_2^2 g$, which equals
$\Delta^4 h_1 c g$.

Both $\eta \alpha$ and $\nu \beta$ are detected by the same
element of the $E_\infty$-page, but they are not equal.
The first product is not annihilated by $\tau^2$, while the latter
product is annihilated by $\tau^2$.
In fact, the difference between $\eta \alpha$ and $\nu \beta$
is detected by $\tau^7 \Delta h_1 g^5$.
This phenomenon corresponds to the classical relation
$\nu^2 \nu_4 = \eta \epsilon_4 + \eta_1 \kappabar^4$
\cite[Proposition 9.17]{BR21}.
\end{rem}

\begin{rem}
\label{rem:eta-D^2h2d}
$(65,3,34)$
The chart in \cite{Bau08} shows a hidden $\eta$ extension from
$\Delta^2 h_2 d$ to $\Delta h_1^2 g^2$ in the 66-stem.  According to \cref{def:hiddenext}, this
is not a hidden extension because of the presence of $\Delta h_1 g^2$ 
in higher filtration.

Nevertheless, there is a relevant point here about multiplicative structure.
Because of the presence of $\tau^3 \Delta h_1 g^2$ in higher filtration,
the element $\Delta^2 h_2 d$ detects two homotopy elements.
One of these elements is annihilated by $\eta$, and one is not.
The product $\nu_2 \kappa$ is one of the two homotopy elements that
are detected by $\Delta^2 h_2 d$.  In fact, $\nu_2 \kappa$
is the homotopy element that is not annihilated by $\eta$.
This follows from the hidden $\eta$ extension from 
$\Delta^2 h_2$ to $\tau^2 \Delta c g$ and the hidden
$\kappa$ extension from $\Delta c g$ to $\tau \Delta h_1^2 g^2$.

\end{rem}

\begin{prop}\label{prop:h0ext110}
$(110,2,56)$
There is a hidden $2$ extension from $\De^4 d$ to $\ta^6 \De^2 h_1^2 g^3$.
\end{prop}

\begin{proof}
The proof is a variation on \cref{method:ctau}, in which we use the long
exact sequence
\begin{center}
\begin{tikzcd}[row sep=0.1em]
& \pi_{*,*}\mmf\arrow[r] & \pi_{*,*}\mmf/\tau^2 \arrow[r] &
\pi_{*-1,*+2}\mmf\arrow[r,"\tau^2"]&\pi_{*-1,*}\mmf\\
\end{tikzcd}
\end{center}
induced by the cofiber sequence
\begin{center}
\begin{tikzcd}[row sep=0.1em]
 & \mmf\arrow[r] & \mmf/\tau^2\arrow[r] &
\Sigma^{1,-2}\mmf \arrow[r,"\tau^2"] & \Sigma^{1,0} \mmf .\\
\end{tikzcd}
\end{center}

We will show that there is a hidden $2$ extension
from $\tau^4 \Delta^4 d g^3$ to $\tau^{10} \Delta^2 h_1^2 g^6$.
The desired $2$ extension follows immediately by multiplication
by $\tau^4 g^3$.

Recall from \cref{prop:d13-2D^7h2} that there is a differential
$d_{13}(2 \Delta^7 h_2) = \tau^6 \Delta^4 d g^3$.  Also,
it follows from \cref{prop:d23-D^5h1} that there is a differential
$d_{23}(\Delta^7 h_1^3) = \tau^{11} \Delta^2 h_1^2 g^6$.

Therefore, $\tau^4 \Delta^4 d g^3$ and $\tau^{10} \Delta^2 h_1^2 g^6$
detect elements in $\pi_{170,88} \mmf$ that are annihilated by $\tau^2$.
Hence they have preimages in $\pi_{171,86} \mmf/\tau^2$ under
projection to the top cell.
By inspection, these preimages are $2 \Delta^7 h_2$ and $\tau \Delta^7 h_1^3$.

In the mANss for $\mmf$,
there is a differential $d_5(\Delta^7) = \tau^2 \Delta^6 h_2 g$.
However, in the mANss for $\mmf/\tau^2$,
the element $\tau^2 \Delta^6 h_2 g$ is already zero in the $E_2$-page.
Therefore, $\Delta^7$ is a permanent cycle in the
mANss for $\mmf/\tau^2$.

Recall the hidden $2$ extension 
from $2 h_2$ to $\tau h_1^3$ established in \cref{lemma:h0ext3}.
Multiplication by $\Delta^7$ gives a 
hidden $2$ extension in the mANss $E_\infty$-page for $\mmf/\tau^2$
from $2 \Delta^7 h_2$ to $\tau \Delta^7 h_1^3$.

Finally, apply projection to the top cell to obtain the hidden
$2$ extension from $\tau^4 \Delta^4 d g^3$ to $\tau^{10} \Delta^2 h_1^2 g^6$.
\end{proof}

\begin{prop}
\label{nu-D^2h1^2}
$(50, 2, 26)$ There is a hidden $\nu$ extension
from $\Delta^2 h_1^2$ to $\tau^2 \Delta h_1 c g$.
\end{prop}

\begin{proof}
This follows from $\Delta h_1$ multiplication on the
hidden extension from
$\Delta h_1$ to $\tau^2 c g$ established in \cref{lemma:h2ext25}.
\end{proof}

The next several lemmas establish some Toda brackets that
we will use to deduce further hidden extensions.
All of these Toda brackets are deduced from algebraic information,
i.e., from Massey products in the mANss $E_2$-page.

\begin{lem}\label{lemma:tb-Dc}
$(32,2,17)$
The Toda bracket $\langle  \nu^2, 2,\eta_1\rangle$
is detected by $\Delta c$ and has no indeterminacy.
\end{lem}

\begin{proof}
We have the Massey product $c=\langle h_2^2, h_0, h_1 \rangle$ in the motivic algebraic Novikov $E_{2}$-page \cite{Baer}. The May convergence theorem \cite{May69} \cite[Theorem 4.16]{BK21} implies that 
$c=\langle h_2^2, 2, h_1 \rangle$ in the mANss $E_2$-page.
Multiply by $\Delta$ to obtain
\[
\De c= \langle h_2^2, 2, h_1 \rangle \Delta = 
\langle h_2^2, 2, \De h_1 \rangle.
\]
The second equality holds because there is no indeterminacy by inspection.

There are no crossing differentials, so the Moss convergence theorem \cite[Theorem 1.2]{Mos70} \cite[Theorem 4.16]{BK21} implies that
$\Delta c$ detects the Toda bracket.
By inspection, the bracket has no indeterminacy.
\end{proof}

\begin{lem}\label{lemma:tb-D^5c}
$(128,2,65)$
The Toda bracket $\left\langle  \nu_2^2, 2,\eta_1\right\rangle$
is detected by $\Delta^5 c$ and has no indeterminacy.
\end{lem}

\begin{proof}
As in the proof of \cref{lemma:tb-D^5c}, we have
the Massey product
$c=\langle h_2^2, 2, h_1 \rangle$ in the mANss $E_2$-page.
Multiply by $\Delta^5$ to obtain
\[
\De^5 c= \De^4 \langle h_2^2, 2, h_1 \rangle \Delta = 
\langle \De^4 h_2^2, 2, \De h_1 \rangle.
\]
The second equality holds because there is no indeterminacy by inspection.

There are no crossing differentials, so the Moss convergence theorem \cite[Theorem 1.2]{Mos70} \cite[Theorem 4.16]{BK21} implies that
$\Delta^5 c$ detects the Toda bracket.
By inspection, the bracket has no indeterminacy.
\end{proof}

\begin{lem}
\label{lemma:tb-t2h1dg}
$(35, 7, 21)$
The Toda bracket
$\langle \nu^2, 2, \epsilon \bar \kappa\rangle$
is detected by $h_1 d g$ and has no indeterminacy.
\end{lem}

\begin{proof}
We have the Massey product 
$h_1 d g = \langle h_2^2, h_0, c g\rangle$ in the motivic algebraic Novikov $E_{2}$-page \cite{Baer}.  
The May convergence theorem \cite{May69} \cite[Theorem 4.16]{BK21}
implies that
$h_1 d g = \langle h_2^2, 2, c g\rangle$
in the mANss $E_2$-page.

There are no crossing differentials, so the Moss convergence theorem \cite[Theorem 1.2]{Mos70} \cite[Theorem 4.16]{BK21} implies that
$h_1 d g$ detects the Toda bracket.
By inspection, the bracket has no indeterminacy.
\end{proof}

\begin{lem}
\label{lemma:tb-t2D4h1dg}
$(131, 7, 69)$
The Toda bracket
$\left\langle \nu_2^2, 2, \epsilon \bar \kappa \right\rangle$
is detected by $\Delta^4 h_1 d g$ and has no indeterminacy.
\end{lem}

\begin{proof}
As in the proof of \cref{lemma:tb-t2h1dg}, we have the
Massey product $h_1 d g = \langle h_2^2, 2, c g\rangle$ in the 
mANss $E_2$-page.
Multiply by $\Delta^4$ to obtain 
$$\Delta^4 h_1 d g = 
\Delta^4\langle h_2^2, h_0, c g \rangle =
\langle \Delta^4h_2^2, h_0, c g \rangle.$$
The second equality holds because there is no indeterminacy by inspection.

There are no crossing differentials, so the Moss convergence theorem
\cite[Theorem 1.2]{Mos70} \cite[Theorem 4.16]{BK21} implies that 
$\Delta^4 h_1 d g$ detects the Toda bracket.
By inspection, the bracket has no indeterminacy.
\end{proof}

\begin{prop}
\label{prop:h2ext32}
There are hidden $\nu$ extensions:
\begin{enumerate}
\item
$(32, 2, 17)$
from $\Delta c$ to $\tau^2 h_1 d g$.
\item
$(128, 2, 65)$
from  $\Delta^5 c$ to $\tau^2 \Delta^4 h_1 d g$.
\end{enumerate}
\end{prop}

\begin{proof}
Recall from \Cref{lemma:tb-Dc} that the Toda bracket 
$\langle \nu^2, 2,\eta_1 \rangle$ is detected by $\Delta c$.
We have
$$\langle\nu^2 , 2, \eta_1\rangle\nu  = \langle \nu^2, 2, \nu \cdot \eta_1 \rangle = \langle \nu^2, 2, \tau^2 \epsilon \bar \kappa\rangle.$$
The first equality holds because there is no indeterminacy by inspection.
The second equality follows from the 
hidden $\nu$ extension of \Cref{lemma:h2ext25}.
\Cref{lemma:tb-t2h1dg} implies that $\tau^2 h_1 d g$
detects the last Toda bracket.

The proof for the second hidden extension is nearly identical.
Consider the equalities
$$\langle \nu_2^2 , 2, \eta_1\rangle\nu  = 
\langle \nu_2^2, 2, \nu \cdot \eta_1 \rangle = 
\langle \nu_2^2, 2, \tau^2 \epsilon \bar \kappa\rangle,$$
and use 
\Cref{lemma:tb-D^5c} and \Cref{lemma:tb-t2D4h1dg}.
\end{proof}

\begin{prop}
\label{prop:h2e-D6h13}
There are hidden $\nu$ extensions:
\begin{enumerate}
\item 
$(97, 1, 49)$
from $\Delta^4 h_1$ to  $\tau^9 g^5$.
\item 
$(122, 2, 62)$
from $\Delta^5 h_1^2$ to $ \tau^9 \Delta h_1 g^5$.
\item
$(147, 3, 75)$
from $\Delta^6 h_1^3$ to $\tau^9\Delta^2 h_1^2 g^5$.
\end{enumerate}
\end{prop}

\begin{proof}
We prove the third hidden extension.
Then the first two
hidden extensions follow from multiplication by $\Delta h_1$.

\cref{prop:h0ext110} and \cref{2-D^6h2^2} imply that
there is a hidden $4 \nu$ extension from $\Delta^6 h_2$ to
$\tau^{10} \Delta^2 h_1^2 g^5$.
We also have a hidden $2$ extension from $2 \Delta^6 h_2$ to
$\tau \Delta^6 h_1^3$, as shown in \cref{prop:he-mmf/t}.
It follows that there must be a hidden
$\nu$ extension from $\Delta^6 h_1^3$ to $\tau^9 \Delta^2 h_1^2 g^5$.
\end{proof}

\begin{prop}
\label{prop:epsilon-D^4d}
$(110, 2, 56)$
There is a hidden $\epsilon$ extension from $\Delta^4 d$ to
$\tau \Delta^4 h_1^2 g$.
\end{prop}

\begin{proof}
We showed in \cref{lemma:ext-cd} that there is a hidden
$\epsilon$ extension from $d$ to $\tau h_1^2 g$.
Multiply by $\Delta^4 h_1$ to obtain a hidden
$\epsilon$ extension from $\Delta^4 h_1 d$ to $\tau \Delta^4 h_1^2 g$.
Finally, use $h_1$ multiplication to obtain the hidden
extension on $\Delta^4 d$.
\end{proof}
 
\begin{prop}
\label{prop:h2ext135}
$(135, 3, 69)$
There is a hidden $\nu$ extension
from $\Delta^5 h_1 d$ to $ \tau^3\Delta^4 h_1^2 g^2.$
\end{prop}

\begin{proof}
By \Cref{lemma:toda}, 
the element $\Delta h_1$ detects the Toda bracket
$\langle\eta , \nu, \ta^2 \bar \kappa\rangle$.
Recall from \cref{htpy-name} that $\kappa_4$ is an element of
$\pi_{110,56} \mmf$ that is detected by the permanent 
cycle $\Delta^4 d$.
Then the element $\Delta^5 h_1 d$ detects
$\langle \eta , \nu, \ta^2 \bar \kappa\rangle \kappa_4$.
Now shuffle to obtain
\[
\nu \langle \eta, \nu, \tau^2 \kappabar \rangle  \kappa_4 =
\langle \nv, \eta, \nu \rangle \tau^2 \kappabar \cdot \kappa_4.
\]
Recall from \cref{lemma:todaepsilon} that $\epsilon= \langle \nu,\eta, \nu\rangle$.
Also recall from 
\Cref{prop:epsilon-D^4d} that there is a hidden 
$\epsilon$ extension from $\Delta^4 d$ to $\ta \Delta^4 h_1^2 g$. 
We conclude that 
$\epsilon \cdot \tau^2 \kappabar \cdot \kappa_4$
is detected by $\tau^3 \Delta^4 h_1^2 g^2$.
\end{proof}

\section{The elements $\nu_k$}
\label{sec:nu_k}

The multiplicative structure of classical $\pi_* \tmf$ at the prime $2$
has been completely computed, with one exception \cite[p.\ 19]{BR21}.
We will use the mANss for $\mmf$ in order to resolve this last piece
of $2$-primary multiplicative structure.

As discussed in \cref{edge-homomorphism}, our choices of
homotopy elements are not necessarily strictly compatible
with the choices in \cite{BR21}.  However,
our choices do agree up to multiples of certain powers of $2$.
Our computations 
below in \cref{nu_k-M}, \cref{nu_i-nu_j},
\cref{cor:nu4-nu6}, \cref{nu3-product}, and \cref{nu-D4}
lie in groups of order at most 8,
so the possible discrepancies are irrelevant.

We will frequently multiply by the element
$\tau \kappabar$ in $\pi_{20,11} \mmf$
in order to detect elements and relations.
Beware that multiplication by $\tau \kappabar$ is not injective
in general.  However, in all degrees that we study,
multiplication by
$\tau \kappabar$ is in fact an isomorphism.

Recall the projection $q: \mmf/\tau \rightarrow \mmf$
to the top cell that was discussed in detail in
\cref{subsctn:inclusion-projection}.  We will rely heavily
on this map in order to transfer the algebraic
information in $\pi_{*,*} \mmf/\tau$ into homotopical
information about $\pi_{*,*} \mmf$.

\begin{lem}
\label{q(Delta)-detect}
The element $q(\Delta^{k+1})$ of $\pi_{*,*} \mmf$ is
detected by $-(k+1) \tau \Delta^k h_2 g$ in
Adams-Novikov filtration $5$.
\end{lem}

\begin{proof}
If $k+1$ is not a multiple of $4$, then we have the non-zero
differential $d_5(\Delta^{k+1}) = (k+1) \tau^2 \Delta^k h_2 g$.
\cref{topcell} implies that 
$q(\Delta^{k+1})$ is detected by $-(k+1) \tau \Delta^k h_2 g$.

If $k+1$ is congruent to $4$ modulo $8$, then
we have the non-zero differential
$d_7(\Delta^{k+1}) = \tau^3 \Delta^k h_1^3 g$.
\cref{topcell} implies that 
$q(\Delta^{k+1})$ is detected by $\tau^2 \Delta^k h_1^3 g$
in filtration $7$.  This implies that $q(\Delta^{k+1})$ 
is detected by zero in filtration $5$.

If $k+1$ is a multiple of $8$, then $\Delta^k$ is a permanent cycle,
so $q(\Delta^{k+1})$ equals zero.  This implies that
$q(\Delta^{k+1})$ is detected by zero in filtration $5$.
\end{proof}

\begin{rem}
\label{rem:q(Delta)-detect}
For uniformity,
we have stated \cref{q(Delta)-detect} for all values of $k$.
As shown in the proof of the lemma, there are in fact three
cases, depending on the value of $k$.  If
$k+1$ is not a multiple of $4$, then 
$-(k+1) \tau \Delta^k h_2 g$ is a non-zero element in
the mANss $E_\infty$-page.

On the other hand, if $k+1$ is a multiple of $4$,
then 
$-(k+1) \tau \Delta^k h_2 g$ is zero in the $E_\infty$-page
since $\tau \Delta^k h_2 g$ is an element of order $4$.
In these cases, the lemma says that
$q(\Delta^{k+1})$ is detected by zero in filtration $5$.
In other words, $q(\Delta^{k+1})$ is detected in filtration
strictly greater than 5, if it is non-zero.
In fact, $q(\Delta^{k+1})$ is detected by
$\tau^2 \Delta^k h_1^3 g$ in filtration $7$
when $k+1$ is congruent to $4$ modulo $8$.
Also, $q(\Delta^{k+1})$ is zero when $k+1$ is a multiple of $8$
because $\Delta^{k+1}$ is a permanent cycle.
\end{rem}

\begin{lem}
\label{q(Delta)-kappabar}
The element
$q(\Delta^{k+1})$ is a multiple of $\tau \kappabar$.
\end{lem}

\begin{proof}
\cref{q(Delta)-detect} shows that
$q(\Delta^{k+1})$ is detected by $-(k+1) \tau \Delta^k h_2 g$.
By inspection, all possible values of $q(\Delta^{k+1})$ are multiples of $\tau \kappabar$.
\end{proof}

\begin{defn}
\label{defn:nu_i}
Let $\nu_k$ be the element of $\pi_{24k+3,12k+2} \mmf$
such that $q(\Delta^{k+1})$ equals $-\tau \kappabar \cdot \nu_k$.
\end{defn}

Note that $\nu_k$ exists because of \cref{q(Delta)-kappabar}.
Multiplication by
$\tau \kappabar$ is an isomorphism in the relevant degrees,
so $\nu_k$ is specified uniquely.
We choose a minus sign in the defining formula of
\cref{defn:nu_i} for later convenience.

\begin{rem}
\label{honorary}
Bruner and Rognes consider $\nu_3$ and $\nu_7$ to be ``honorary" members
of the family of elements $\nu_k$.  They are not multiplicative generators;
$\nu_3$ is non-zero but decomposable, and $\nu_7$ equals zero.
\cref{defn:nu_i} also implies that $\nu_7$ is zero.  This follows from
the observation that $q(\Delta^8)$ equals zero since
$\Delta^8$ is a permanent cycle.
\end{rem}

The careful reader will note that the elements
$\nu_k$ were already partially defined in \cref{htpy-name} in \cref{subsctn:homotopy-elements}.  The following lemma shows that the two approaches to
$\nu_k$ are compatible.  
\cref{htpy-name} leaves some ambiguity in the definition of $\nu_k$,
and \cref{defn:nu_i} resolves that ambiguity.

\begin{lem}
\label{nu_i-detect}
The element $\nu_k$ is
detected by $(k+1) \Delta^k h_2$ in Adams-Novikov filtration $1$.
\end{lem}

\begin{proof}
\cref{q(Delta)-detect} determines the mANss $E_\infty$-page elements
that detect $q(\Delta^{k+1})$.
Then \cref{defn:nu_i} means that
$-\tau \kappabar \cdot \nu_k$ is detected by
those same elements.  Multiplication by $\tau g$ is
an isomorphism in the relevant degrees,
so the detecting elements for $\nu_k$ are then determined.
\end{proof}

\begin{rem}
\label{rem:nu_i-detect}
Similarly to \cref{rem:q(Delta)-detect},
\cref{nu_i-detect} includes three cases.
If $k+1$ is not a multiple of $4$, then
$(k+1) \Delta^k h_2$ is a non-zero element of the mANss
$E_\infty$-page.
If $k+1$ is a multiple of $4$, then
$(k+1) \Delta^k h_2$ is zero since $\Delta^k h_2$ is an
element of order $4$.
This means that $\nu_k$ is detected in filtration strictly
greater than $1$, if it is non-zero.
In fact, $\nu_k$ is detected by
$\tau \Delta^k h_1^3$ in filtration $3$ if $k+1$ is congruent to $4$ modulo $8$,
and $\nu_k$ is zero if $k+1$ is a multiple of $8$.
\end{rem}

\begin{rem}
\label{nu-defn}
Earlier in \cref{h2-choice}, we chose $h_2$ so that it detects the
element $\nu$.  \cref{nu_i-detect} shows that $\nu_0$
is also detected by $h_2$, but that does not guarantee that it equals
$\nu$ because of the presence of $\tau h_1^3$ in higher filtration.
We can only conclude that $\nu$ and $\nu_0$ are equal up to multiples
of $4$.  

If $\nu$ equals $5 \nu_0$, then we compute that
\[
q(5 \Delta) = - 5 \tau \kappabar \cdot \nu_0 = -\tau \kappabar \cdot \nu.
\]
So we may replace $\Delta$ by $5 \Delta$, if necessary,
and assume without loss of generality that $\nu_0$ equals $\nu$.
This replacement is compatible with our previous choice of
$\Delta$ in \cref{Delta-defn}, which specified $\Delta$ only 
up to multiples of $4$.
\end{rem}

\begin{prop}
\label{nu_k-M}
$\nu_{k+8} = \nu_k \cdot M$.
\end{prop}

\begin{proof}
Using \cref{q-module-map}, we have
\[
q(\Delta^{k+9}) = 
q(\Delta^{k+1} \cdot \Delta^8) =
q(\Delta^{k+1} \cdot i(M)) = 
q(\Delta^{k+1}) \cdot M =
- \tau \kappabar \cdot \nu_k \cdot M.
\]
Here we are using that $i(M) = \Delta^8$,
which is equivalent to the definition that
$M$ is detected by $\Delta^8$ (see \cref{htpy-name}).

On the other hand,
$q(\Delta^{k+9})$ equals $-\tau \kappabar \cdot \nu_{k+8}$
by \cref{defn:nu_i}.
Finally, multiplication by $-\tau \kappabar$
is an isomorphism in the relevant degrees.
\end{proof}

\cref{nu_k-M} means that for practical purposes, we only need to consider
the elements $\nu_k$ for $0 \leq k \leq 7$.

\begin{thm}
\label{nu_i-nu_j}
\[
\nu_j \nu_k = (k+1) \nu_{j+k} \nu_0.
\]
\end{thm}

\begin{proof}
The proof splits into two cases, depending on whether
$k+1$ is a multiple of $4$.  First, we handle the (more interesting)
situation when $k+1$ is not a multiple of $4$.
We address the case when $k+1$ is a multiple of $4$
below in a separate \cref{nu3-product}.  The proof techniques for the
two cases are similar, but the details are somewhat different.

Multiplication by $\tau \kappabar$ is an isomorphism in the relevant degrees,
so it suffices to establish our relation after multiplication by $\tau \kappabar$.

Using \cref{q-module-map}, 
we have
\begin{align*}
& q( (k+1) \Delta^{j+k+1} h_2 ) =
q(\Delta^{j+k+1} \cdot (k+1) h_2 ) =
q (\Delta^{j+k+1} \cdot i( (k+1) \nu_0 )) = \\
& = 
q(\Delta^{j+k+1}) \cdot (k+1) \nu_0 =
-\tau \kappabar \cdot \nu_{j+k} \cdot (k+1) \nu_0.
\end{align*}
Here we are using that $i((k+1)\nu_0) = (k+1)h_2$; in other words,
$(k+1)\nu_0$ is detected by $(k+1) h_2$.
This requires that $k+1$ is not a multiple of $4$.
Otherwise, $(k+1) \nu_0$ is a multiple of $\tau$,
and $i((k+1) \nu_0)$ is zero.

We will now compute $q( (k+1) \Delta^{j+k+1} h_2 )$ another way.
We have $i(\nu_k) = (k+1) \Delta^k h_2$; in other words,
$\nu_k$ is detected by the non-zero element $(k+1) \Delta^k h_2$,
as shown in \cref{nu_i-detect}.  
This requires that $k+1$ is not a multiple of $4$.
Otherwise, $\nu_k$ is a multiple of $\tau$,
and $i(\nu_k)$ is zero.

Then we have
\[
q( (k+1) \Delta^{j+k+1} h_2 ) = 
q (\Delta^{j+1} \cdot (k+1) \Delta^k h_2) =
q( \Delta^{j+1} \cdot i(\nu_k) ) = 
q(\Delta^{j+1}) \cdot \nu_k =
- \tau \kappabar \cdot \nu_j \cdot \nu_k.
\]
\end{proof}

\begin{rem}
\label{nu_i-nu_j-commute}
The exact form of the equation in \cref{nu_i-nu_j} is guided by
the structure of our proof.  One could also write
\[
\nu_i \nu_j = (i+1) \nu \nu_{i+j},
\]
which more closely aligns with the notation in \cite{BR21}.
All of the elements $\nu_k$ are in odd stems, so they pairwise 
anti-commute.  
\end{rem}

\begin{cor}
\label{cor:nu4-nu6}
$(246,2,124)$
$\nu_4 \nu_6 = \nu \nu_2 M$.
\end{cor}

\begin{proof}
\cref{nu_i-nu_j} implies that
$\nu_4 \nu_6$ equals $7 \nu_{10} \nu_0$, which equals
$-7 \nu_0 \nu_{10}$ by graded com\-mu\-ta\-tiv\-i\-ty.
By \cref{nu-defn} and \cref{nu_k-M}, the latter expression
equals $-7 \nu \nu_2 M$.  Finally,
$\nu \nu_2 M$ belongs to a group of order $4$, so
$-7 \nu \nu_2 M$ equals $\nu \nu_2 M$.
\end{proof}

We now return to the case of \cref{nu_i-nu_j}
in which $k+1$ is a multiple of $4$.

\begin{prop}
\label{nu3-product}
If $k+1$ is a multiple of $4$, then 
$\nu_j \cdot \nu_k = (k+1) \nu_{j+k} \nu_0$.
\end{prop}

\begin{proof}
First, let $k+1$ be a multiple of $8$, so $\nu_k$ is zero.
The element $\nu_{j+k} \nu_0$ belongs to a group whose order divides $8$, so 
$(k+1) \nu_{j+k} \nu_0$ is zero.  In other words, the equality holds
because both sides are zero.

Next, let $k+1$ be congruent to $4$ modulo $8$.
Let $\alpha$ be an element of $\pi_{*,*}\mmf$ that is detected by
$\Delta^k h_1^3$.  The element $\nu_k$ is detected by $\tau \Delta^k h_1^3$,
according to \cref{rem:nu_i-detect}.  Since there are no elements
in higher filtration, we can conclude that $\nu_k$ equals $\tau \alpha$.
We have 
\[
q(\Delta^{j+k+1} h_1^3) = q(\Delta^{j+1} \cdot \Delta^k h_1^3) = 
q(\Delta^{j+1} \cdot i(\alpha)) = q(\Delta^{j+1}) \cdot \alpha =
- \tau \kappabar \cdot \nu_j \cdot \alpha = -\kappabar \cdot \nu_j \cdot \nu_k.
\]

Now we add the assumption that $j+1$ is not congruent to $4$ modulo $8$.
Given the assumption that $k+1$ is congruent to $4$ modulo $8$, we get that
$j+k+1$ is not congruent to $7$ modulo $8$.
Then $\Delta^{j+k+1} h_1^3$ is a permanent cycle, so
$q(\Delta^{j+k+1} h_1^3)$ is zero.
Together with the computation in the previous paragraph,
this implies that $\nu_j \cdot \nu_k$ is zero since
multiplication by $\kappabar$ is an isomorphism in the relevant degrees.
Note also that $(k+1) \nu_{j+k} \nu_0$ is zero because
it belongs to a group whose order divides $4$.

Finally, we must consider the case when $j+1$ is congruent to $4$ modulo $8$,
i.e., that $j+k+1$ is congruent to $7$ modulo $8$.
Then $q(\Delta^{j+k+1} h_1^3)$ is detected by
$\tau^{10} \Delta^{j+k-4} h_1^2 g^6$ because of \cref{topcell}
and the differential $d_{23}(\Delta^{j+k+1} h_1^3) = \tau^{11} \Delta^{j+k-4} h_1^2 g^6$.
This means that $- \kappabar \cdot \nu_j \cdot \nu_k$
is detected by $\tau^{10} \Delta^{j+k-4} h_1^2 g^6$.
It follows that $\nu_j \cdot \nu_k$ is detected by
$\tau^{10} \Delta^{j+k-4} h_1^2 g^5$.
Finally, this latter element also detects $(k+1) \nu_{j+k} \nu_0$ because
of the hidden $2$ extensions in the $150$-stem and their
multiples under $\Delta^8$ multiplication (see \cref{hidh0}).
\end{proof}

\begin{rem}
As shown in the proof, most cases of \cref{nu3-product} hold because
both sides of the equation are zero.  Both sides of the equation are non-zero
precisely when $j+1$ and $k+1$ are congruent to $4$ modulo $8$.
\end{rem}

Bruner and Rognes establish some relations that reduce the ambiguity
in their definitions of $\nu_k$.  Finally, we will show that our elements
defined in \cref{defn:nu_i} satisfy those same relations.
We have already discussed
the choice of $\nu_0$ in \cref{nu-defn}.  
The only additional requirements are the relations
\begin{align*}
& \nu_0 D_4 = 2 \nu_4 \\
& \nu_1 \nu_5 = 2 \nu_0 \nu_6 \\
& \nu_2 \nu_4 = 3 \nu_0 \nu_6.
\end{align*}
The first formula is proved in \cref{nu-D4}, while the last two
are specific instances of \cref{nu_i-nu_j}.

\begin{prop}
\label{nu-D4}
$(99,1,50)$
$\nu_0 D_4 = 2 \nu_4$.
\end{prop}

\begin{proof}
Because of \cref{nu_i-detect},
both products are detected by $2 \Delta^4 h_2$.
However, they are not necessarily equal because of the
presence of $\tau \Delta^4 h_1^3$ in higher filtration.
We will show that $\tau \kappabar \cdot \nu D_4$ equals
$\tau \kappabar \cdot 2 \nu_4$.  Our desired relation follows
immediately because multiplication by $\tau \kappabar$ 
is an isomorphism in the relevant degree.

Using \cref{q-module-map},
we have
\[
q(2 \Delta^5) = q(\Delta \cdot 2 \Delta^4) = 
q(\Delta \cdot i(D_4)) = q(\Delta) \cdot D_4 =
- \tau \kappabar \cdot \nu \cdot D_4.
\]
Here we are using that $i(D_4) = 2 \Delta^4$, which is equivalent to
the definition that $D_4$ is detected by $2 \Delta^4$
(see \cref{htpy-name}).
On the other hand, we also have
\[
q(2\Delta^5) = q(\Delta^5 \cdot 2) = q(\Delta^5 \cdot i(2)) =   q(\Delta^5) \cdot 2 = - \tau \kappabar \cdot \nu_4 \cdot 2.
\]
\end{proof}

\newpage
\section{Tables}
\label{sec:tables}

\begin{longtable}{lllll}
\caption{Adams-Novikov differentials
\label{diff} 
} \\
\toprule
$(s,f,w)$ & $x$ & $r$ & $d_r(x)$& proof \\
\midrule \endfirsthead
\caption[]{Adams-Novikov differentials} \\
\toprule
$(s,f,w)$ & $x$ & $r$ & $d_r(x)$& \\
\midrule \endhead
\bottomrule \endfoot
$(5, 1, 3)$   & $h_1v_1^2$   & 3 & $\tau h_1^4$&\Cref{prop:d3}\\
$(12, 0, 6)$   & $4a$   & 3 & $\tau P h_1^3$&\Cref{prop:d3bopattern}\\
$(24, 0, 12)$ & $\Delta$     & 5 & $\tau^2 h_2g$ & \Cref{prop:d5_D} \\
$(24, 0, 12)$ & $4\Delta$    & 7 & $\tau^3 h_1^3g$ &\Cref{prop:d7}\\
$(48, 0, 24)$ & $2\Delta^2$  & 7 &$\tau^3 \Delta h_1^3 g$ &\Cref{prop:d7}\\
$(96, 0, 48)$ & $\Delta^4$   & 7 & $\tau^3 \Delta^3 h_1^3 g$&\Cref{prop:d7two} \\
$(49, 1, 25)$ & $\Delta^2 h_1$ & 9 & $\tau^4 cg^2$  &\Cref{prop:d9two}\\
$(56, 2, 29)$ & $\Delta^2 c$ & 9 & $\tau^4 h_1 d g^2$ &\Cref{prop:d9-D^2c}\\
$(73, 1, 37)$ & $\Delta^3 h_1$& 9 & $\tau^4 \Delta c g^2$&\Cref{prop:d9two}\\
$(80, 2, 41)$ & $\Delta^3 c$ & 9 & $\tau^4 \Delta h_1 d g^2$ & \cref{prop:d9one} \\
$(145, 1, 73)$ & $\Delta^6 h_1$ & 9 &  $\tau^4 \Delta^4 c g^2$ &\Cref{prop:d9two} \\
$(169, 1, 85)$ & $\Delta^7 h_1$ & 9 &  $\tau^4 \Delta^5 c g^2$ &\Cref{prop:d9two}\\
$(152, 2, 77)$ & $\Delta^6 c$ & 9 & $\tau^4 \Delta^4 h_1 d g^2$ &\Cref{prop:d9-D^6c}\\
$(176, 2, 89)$ & $\Delta^7 c$ & 9 &  $\tau^4 \Delta^5 h_1 d g^2$ &\Cref{prop:d9one}\\
$(62, 2, 32)$ & $\Delta^2 d$ & 11 & $\tau^5 h_1g^3$ &\Cref{prop:d11}\\
$(158, 2, 80)$ & $\Delta^6 d$ & 11 &  $\tau^5 \Delta^4 h_1g^3$ &\Cref{prop:d11}\\
$(75, 1, 38)$ & $2\Delta^3 h_2$ & 13 & $\tau^6 dg^3$ &\Cref{prop:d13-2D^3h2}\\
$(81, 3, 42)$ & $\Delta^3 h_1 c$ & 13 & $2 \tau^6 g^4$&\Cref{prop:d13two}\\
$(171, 1, 86)$ & $2\Delta^7 h_2$ & 13 & $\tau^6 \Delta^4 dg^3$ &\Cref{prop:d13-2D^7h2}\\
$(177, 3, 90)$ & $\Delta^7 h_1 c$ & 13 & $2 \tau^6 \Delta^4 g^4$&\Cref{prop:d13two}\\
$(121, 1, 61)$ & $\Delta^5 h_1$ & 23 & $\tau^{11} g^6$&\Cref{prop:d23-D^5h1}
\end{longtable}

\begin{longtable}{llll}
\caption{Hidden $2$ extensions
\label{hidh0} 
} \\
\toprule
$(s,f,w)$ & source & target& proof \\
\midrule \endfirsthead
\caption[]{Hidden $2$ extensions} \\
\toprule
$(s,f,w)$ & source & target \\
\midrule \endhead
\bottomrule \endfoot
$(3, 1, 2)$    & $2h_2$          &  $\tau h_1^3$ & \Cref{lemma:h0ext3}\\
$(27, 1, 14)$  & $2\Delta h_2$   &  $\tau \Delta h_1^3$ &\Cref{lem:2-Dh0h2}\\
$(51, 1, 26)$  & $2\Delta^2 h_2$ &  $\tau \Delta^2 h_1^3$ &\Cref{prop:he-mmf/t}\\
$(54, 2, 28)$  & $\Delta^2 h_2^2$ & $\tau^4 d g^2$ & \Cref{exmp:hidextn} \\
$(99, 1, 50)$  & $2\Delta^4 h_2$ & $\tau \Delta^4 h_1^3$ &\Cref{prop:he-mmf/t}\\
$(110, 2, 56)$ & $\Delta^4 d$   &$\tau^6 \Delta^2 h_1^2 g^3$ &\Cref{prop:h0ext110} \\
$(123, 1, 62)$ & $2\Delta^5 h_2$ &  $\tau \Delta^5 h_1^3$ &\Cref{prop:he-mmf/t}\\
$(147, 1, 74)$ & $2\Delta^6 h_2$ &  $\tau \Delta^6 h_1^3$ & \Cref{prop:he-mmf/t}\\
$(150, 2, 76)$ & $\Delta^6 h_2^2$ & $\tau^4 \Delta^4 d g^2$ &\Cref{prop:he-mmf/t}\\
\end{longtable}

\newpage

\begin{longtable}{llll}
\caption{Hidden $\eta$ extensions
\label{hidh1} 
} \\
\toprule
$(s,f,w)$ & source & target & proof \\
\midrule \endfirsthead
\caption[]{Hidden $\eta$ extensions} \\
\toprule
$(s,f,w)$ & source & target \\
\midrule \endhead
\bottomrule \endfoot
$(27, 1, 14)$  & $2\Delta h_2$   &  $\tau^2 c g$ & \Cref{lemma:h1ext27}\\
$(39, 3, 21)$  & $\Delta h_1 d$ &  $2 \tau^2 g^2$ & \Cref{lemma:h1ext39} \\
$(51, 1, 26)$  & $\Delta^2 h_2$ & $\tau^2 \Delta c g$ &  \Cref{prop:he-mmf/t}\\
$(99, 1, 50)$ & $\Delta^4 h_2$   & $\tau^9 g^5$& \Cref{prop:he-mmf/t} \\
$(123, 1, 62)$ & $2\Delta^5 h_2$ &  $\tau^2 \Delta^4 c g$ &\Cref{prop:he-mmf/t} \\
$(124, 6, 63)$ & $\ta^2\Delta^4 c g$ &  $\ta^9 \Delta h_1 g^5$ &\Cref{prop:he-mmf/t}\\
$(129, 3, 66)$ & $\Delta^5 h_1 c$ & $\tau^7 \Delta^2 h_1^2 g^4$ &\Cref{prop:he-mmf/t} \\
$(135, 3, 69)$ & $\Delta^5 h_1 d$ & $2 \tau^2 \Delta^4 g^2 $ &  \Cref{prop:he-mmf/t} \\
$(147, 1, 74)$ & $\Delta^6 h_2$   &$\tau^2 \Delta^5 c g$ &\Cref{prop:he-mmf/t} \\
$(161, 3, 82)$ & $\Delta^6 h_2 d$ & $\tau^3 \Delta^5 h_1^2 g^2$ &\Cref{prop:he-mmf/t}\\
\end{longtable}

\begin{longtable}{llll}
\caption{Hidden $\nu$ extensions
\label{hidh2} 
} \\
\toprule
$(s,f,w)$ & source & target & proof \\
\midrule \endfirsthead
\caption[]{Hidden $\nu$ extensions} \\
\toprule
$(s,f,w)$ & source & target \\
\midrule \endhead
\bottomrule \endfoot
$(0, 0, 0)$ & $4$ & $\tau h_1^3$ & \cref{prop:he-mmf/t} \\
$(25, 1, 13)$   & $\Delta h_1$          &  $\tau^2 c g$ &  \Cref{lemma:h2ext25}\\
$(32, 2, 17)$  & $\Delta c$   &  $\tau^2 h_1 d g$ & \Cref{prop:h2ext32}\\
$(39, 3, 21)$  & $\Delta h_1 d$ &$\tau^3 h_1^2 g^2$  &\Cref{lemma:h2ext39}\\
$(48, 0, 24)$  & $4\Delta^2$ & $\tau \Delta^2 h_1^3$ & \cref{prop:he-mmf/t}   \\
$(50, 2, 26)$  & $\Delta^2 h_1^2$ & $\tau^2 \Delta h_1 c g$ & \cref{nu-D^2h1^2} \\
$(51, 1, 26)$  & $2\Delta^2 h_2 $ & $\tau^4 d g^2$ & \cref{prop:he-mmf/t}\\
$(57, 3, 30)$ & $\Delta^2 h_2^3$   & $2 \tau^4 g^3$ & \Cref{prop:he-mmf/t}\\
$(96, 0, 48)$ & $4\Delta^4 $ &  $\tau \Delta^4 h_1^3$  & \cref{prop:he-mmf/t} \\
$(97, 1, 49)$ & $\Delta^4 h_1$ &  $\tau^9 g^5$ & \Cref{prop:h2e-D6h13}\\
$(122,2, 62)$ & $\Delta^5 h_1^2$& $\tau^9 \Delta h_1 g^5 $ & \Cref{prop:h2e-D6h13} \\
$(128, 2, 65)$ & $\Delta^5 c$ & $\tau^2 \Delta^4 h_1 d g$ & \Cref{prop:h2ext32} \\
$(135, 3, 69)$  & $\Delta^5 h_1 d$ &  $\tau^3 \Delta^4 h_1^2 g^2$ &\Cref{prop:h2ext135}\\
$(144, 0, 72)$ & $4\Delta^6 $ &  $\tau \Delta^6 h_1^3$ & \Cref{prop:he-mmf/t} \\
$(147, 1, 74)$ & $2\Delta^6 h_2$   & $\tau^4 \Delta^4 d g^2$ & \Cref{prop:he-mmf/t} \\
$(147, 3, 75)$ & $\Delta^6 h_1^3$   & $\tau^9 \Delta^2 h_1^2 g^5$ & \Cref{prop:h2e-D6h13}\\
$(153, 3, 78)$  & $\Delta^6 h_2^3$ &  $2 \tau^4 \Delta^4 g^3$ & \Cref{prop:he-mmf/t}\\
\end{longtable}

\begin{longtable}{llllll}
\caption{Some Toda brackets
\label{brac} 
} \\
\toprule
$(s,f,w)$ & Toda bracket & detected by & indet & proof & used in  \\
\midrule \endfirsthead \\
\caption[]{Some Toda brackets} \\
\toprule
$(s,f,w)$ & Toda bracket & detected by & indet & proof & used in  \\
\midrule \endhead
\bottomrule \endfoot
$(8, 2, 5)$ & $\langle \nu, \eta, \nu \rangle$ & $c$ & $0$ & \cref{lemma:todaepsilon} & \ref{lemma:h2ext25}, \ref{prop:h2ext135} \\
$(25, 1, 13)$ & $\langle  \eta, \nu, \tau^2 \kappabar \rangle$ & $\Delta h_1$ & $P^3 h_1$ & \cref{lemma:toda} & \ref{lemma:h2ext25}, \ref{lemma:h1ext27}, \ref{prop:h2ext135} \\
$(32, 2, 17)$ & $\langle \nu^2, 2, \eta_1 \rangle$ & $\Delta c$ & $0$ & \cref{lemma:tb-Dc} & \ref{prop:h2ext32} \\
$(128, 2, 65)$ & $\langle \nu_2^2, 2, \eta_1 \rangle$ & $\Delta^5 c$ & $0$ & \cref{lemma:tb-D^5c} & \ref{prop:h2ext32} \\
$(35, 7, 21)$ & $\langle \nu^2, 2, \epsilon \kappabar \rangle$ & $h_1 d g$ & $0$ & \cref{lemma:tb-t2h1dg} & \ref{prop:h2ext32} \\
$(131, 7, 69)$ & $\langle \nu_2^2, 2, \epsilon \kappabar \rangle$ & $\Delta^4 h_1 d g$ & $0$ & \cref{lemma:tb-t2D4h1dg} & \ref{prop:h2ext32} \\
\end{longtable}

\section{Charts}
\label{sec:charts}

The following charts display
the $E_2$-page, $E_9$-page, and
$E_\infty$-page of the mANss for $\mmf$. Each of these pages is
free as a module over $\mathbb{Z}[\Delta^8]$, 
where $\Delta^8$ is a class in the $192$-stem.
For legibility, we display the $v_1$-periodic elements on
separate charts.  See \cref{sec:bo} for discussion of $v_1$-periodicity.
To obtain the full $E_2$-page, one must superimpose
Figures \ref{fig:E2v1} and \ref{fig:E2d7}.  To obtain
the full $E_\infty$-page, one must superimpose
Figures \ref{fig:E4v1} and \ref{fig:E-infty}.

We describe each chart in slightly more detail.
\begin{itemize}
\item
\cref{fig:E2v1} shows the $v_1$-periodic portion of the mANss $E_2$-page, together with all differentials that are supported by the displayed elements.
\item
\cref{fig:E4v1} shows the $v_1$-periodic portion of the mANss $E_\infty$-page.
\item
\cref{fig:E2d7} shows the non-$v_1$-periodic portion of the mANss $E_2$-page, together with all $d_3$, $d_5$, and $d_7$ differentials
that are supported by the displayed elements.
\item
\cref{fig:E9d23} 
shows the non-$v_1$-periodic portion of the mANss $E_9$-page, together with all differentials
that are supported by the displayed elements.
\item
\cref{fig:E-infty} shows the non-$v_1$-periodic portion of the mANss $E_\infty$-page, together with all hidden extensions by $2$,
$\eta$, and $\nu$.
\end{itemize}

\subsection{Elements}
\label{sctn:elements}

For each fixed stem and filtration, the mANss consists of a $\Z[\tau]$-module.
We use a graphical notation to describe these modules.  Our notation
represents the associated graded object of a filtration that is related
to the powers of $2$.
\begin{itemize}
\item
An open box 
\begin{tikzpicture}
		[scale=1.10,
		>=stealth,
		label distance=0,
		label position=below,
		every label/.style={
			inner sep=0,
			scale=0.66},
		every path/.style={thick},
		sq/.style={rectangle, 
			inner sep=0,
			draw={gray},
			minimum size=0.10cm,
			scale=1.10*2},
		]
		\node[sq] at (48.0,0.0) {};
\end{tikzpicture}
indicates a copy of $\Z[\tau]$ in the associated graded object.
\item
A solid gray dot {\color{gray} $\bullet$}
indicates a copy of $\F_2[\tau]$  in the associated graded object. 
\item
A solid colored dot indicates a copy of $\F_2[\tau]/\tau^r$ in the associated
graded object.  The value of $r$ is encoded in the color of the dot,
as shown in \cref{table:chartcolor}.
\item
Short vertical lines indicate extensions by $2$.
\end{itemize}

Our graphical notation has the advantages of flexibility, compactness,
and convenience.  We illustrate with two examples.

\begin{example}
In \cref{fig:E2d7} at degree $(48,0)$, one sees
\begin{tikzpicture}
		[scale=1.10,
		>=stealth,
		label distance=0,
		label position=below,
		every label/.style={
			inner sep=0,
			scale=0.66},
		every path/.style={thick},
		dashed/.style={dash pattern=on 5pt off 2.5pt},
		dot/.style={circle,
			inner sep=0,
			minimum size=0.10cm,
			scale=1.10},
		sq/.style={rectangle, 
			inner sep=0,
			draw={gray},
			minimum size=0.10cm,
			scale=1.10*2},
		tau0/.style={
			draw={gray},
			fill={gray}},
		tau0extn/.style={draw={gray}},
		h0/.style={draw={gray}}]
            \draw[h0, tau0extn] (48.0,0.0+0) -- (48.0,0.24+0);
            \draw[h0, tau0extn] (48.0,0.24+0) -- (48.0,0.48+0);
            \node[dot, tau0, label=$ $] at (48.0,0.0) {};
		\node[dot, tau0, label=$ $] at (48.0,0.24) {};
		\node[sq, label=$ $] at (48.0,0.48) {};
\end{tikzpicture}.
This notation indicates a copy of $\Z[\tau]$.  More precisely, it
represents the filtration 
$4\Z[\tau] \subseteq 2\Z[\tau] \subseteq \Z[\tau]$ whose filtration quotients
are $\Z[\tau]$, $\F_2[\tau]$, and $\F_2[\tau]$.
This particular filtration is relevant for our mANss computation because
$2\Z[\tau]$ is the subgroup of $d_5$ cycles, and
$4\Z[\tau]$ is the subgroup of $d_7$ cycles.
\end{example}

\begin{example}
In \cref{fig:E-infty} at degree $(120,24)$, one sees
		\begin{tikzpicture}
		[scale=1.10,
		>=stealth,
		label distance=0,
		label position=below,
		every label/.style={
			inner sep=0,
			scale=0.66},
		every path/.style={thick},
		dashed/.style={dash pattern=on 5pt off 2.5pt},
		dot/.style={circle,
			inner sep=0,
			minimum size=0.10cm,
			scale=1.10},
		sq/.style={rectangle, 
			inner sep=0,
			draw={gray},
			minimum size=0.10cm,
			scale=1.10*2},
		tau2/.style={
			draw={blue},
			fill={blue}},
		tau6/.style={
			draw={magenta},
			fill={magenta}},
		tau11/.style={
			draw={orange},
			fill={orange}},
            tau2extn/.style={draw={blue}},
		tau6extn/.style={draw={magenta}},
		h0/.style={draw={gray}}]
			\draw[h0, tau6extn] (120.0,24.0+0) -- (120.0,24.24+0);
			\draw[h0, tau2extn] (120.0,24.24+0) -- (120.0,24.48+0);
			\node[dot, tau11, label=$ $] at (120.0,24.0) {};
			\node[dot, tau6, label=$ $] at (120.0,24.24) {};
			\node[dot, tau2, label=$ $] at (120.0,24.48) {};
    \end{tikzpicture}.
This notation indicates the $\Z[\tau]$-module
\[
\frac{\Z[\tau]}{8, 4 \tau^2, 2 \tau^6, \tau^{11}},
\]
which is somewhat cumbersome to describe in traditional notation.
More precisely, it represents the filtration
\[
\frac{4\Z[\tau]}{8, 4 \tau^2} \subseteq
\frac{2\Z[\tau]}{8, 4 \tau^2, 2 \tau^6} \subseteq
\frac{\Z[\tau]}{8, 4 \tau^2, 2 \tau^6, \tau^{11}  }.
\]
whose filtration quotients are 
$\F_2[\tau]/\tau^2$, $\F_2[\tau]/\tau^6$, and $\F_2[\tau]/\tau^{11}$.
The blue, magenta, and orange dots correspond to these filtration
quotients, as shown in \cref{table:chartcolor}.
\end{example}

\begin{longtable}{ll}
\caption{Color interpretations for elements
\label{table:chartcolor}
} \\
\toprule
		$n$ & color \\
		\midrule \endfirsthead
		\caption[]{Color interpretations for elements}\\
		\toprule
		$n$ & color \\
		\midrule \endhead
		\bottomrule \endfoot
$\F_2[\tau]$ & {\color{gray} $\bullet$ gray } \\
$\F_2[\tau]/\tau$ & {\color{red} $\bullet$ red } \\
$\F_2[\tau]/\tau^2$ & {\color{blue} $\bullet$ blue } \\
$\F_2[\tau]/\tau^3$ & {\color{green} $\bullet$ green } \\
$\F_2[\tau]/\tau^4$ & {\color{cyan} $\bullet$ cyan } \\
$\F_2[\tau]/\tau^5$ & {\color{brown} $\bullet$ brown } \\
$\F_2[\tau]/\tau^6$ & {\color{magenta} $\bullet$ magenta } \\
$\F_2[\tau]/\tau^{11}$ & {\color{orange} $\bullet$ orange } \\
\end{longtable}

\subsection{Differentials}

Lines of negative slope indicate Adams-Novikov differentials.
The differentials are colored according to their lengths,
as described in \cref{table:diffcolor}.
These color choices are compatible with our choice of colors
for $\tau$ torsion in \cref{sctn:elements}, in the following sense.
An Adams-Novikov $d_{2r+1}$ differential always takes the form
$d_{2r+1}(x) = \tau^r y$, and it creates $\tau^r$ torsion in the following
page.  We use matching colors for $d_{2r+1}$ and for $\tau^r$ torsion.

\begin{longtable}{lll}
\caption{Color interpretations for Adams-Novikov differentials
\label{table:diffcolor}
} \\
\toprule
color & slope & $d_r$ \\
\midrule \endfirsthead
\caption[]{Color interpretations for Adams-Novikov differentials}\\
\toprule
color & slope & $d_r$ \\
\midrule \endhead
\bottomrule \endfoot
{\color{red} red} & $-3$ & $d_3$ \\
{\color{blue} blue} & $-5$ & $d_5$ \\
{\color{green} green} & $-7$ & $d_7$ \\
{\color{cyan} cyan} & $-9$ & $d_9$ \\
{\color{brown} brown} & $-11$ & $d_{11}$ \\
{\color{magenta} magenta} & $-13$ & $d_{13}$ \\
{\color{orange} orange} & $-23$ & $d_{23}$ \\
\end{longtable}

\subsection{Extensions}

\begin{itemize}
\item
Solid lines of slope $1$ indicate $h_1$ multiplications.  The colors of
these lines are determined by the $\tau$ torsion of the targets.
\item
Arrows of slope $1$ indicate infinite families of elements that are connected by $h_1$ multiplications.  The colors of the arrows reflect the $\tau$ torsion
of the elements.
\item
Solid lines of slope $1/3$ indicate $h_2$ multiplications.  The colors of
these lines are determined by the $\tau$ torsion of the targets.
\item
Dashed lines indicate hidden extensions by $2$, $\eta$, and $\nu$.
Some of these lines are curved solely for the purpose of legibility.
\item
The colors of dashed lines indicate the $\tau$ torsion of the targets of
the extensions.
For example, the vertical dashed line in the $23$-stem of
\cref{fig:E-infty} is blue because its value $\tau h_1^3 g$ is annihilated
by $\tau^2$.
\end{itemize}

\cref{fig:E-infty} shows an $h_1$ extension and also a hidden
$\eta$ extension on the element $\Delta^4 c g$ in degree $(124,6,65)$.
See \cref{eta-t^2D^4cg} for an explanation.

\KOMAoption{paper}{landscape,8.4in:5.3in}
\KOMAoption{DIV}{last}
\newgeometry{margin=0in,footskip=0.3in}
\begin{figure}[H]
\begin{center}
\makebox[0.95\textwidth]{\includegraphics[trim={0cm, 0cm, 0cm, 0cm},clip,page=1,scale=0.2]{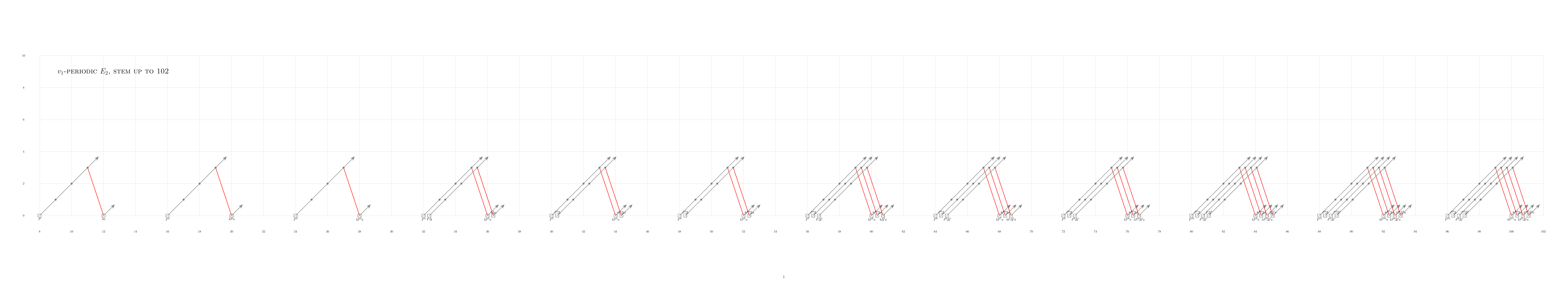}}
\caption{The $v_1$-periodic portion of the $\C$-motivic 
Adams-Novikov $E_2$-page for $\mmf$}
\label{fig:E2v1}
\hfill
\end{center}
\end{figure}

\begin{figure}[H]
\begin{center}
\makebox[0.95\textwidth]{\includegraphics[trim={0cm, 0cm, 0cm, 0cm},clip,page=1,scale=0.2]{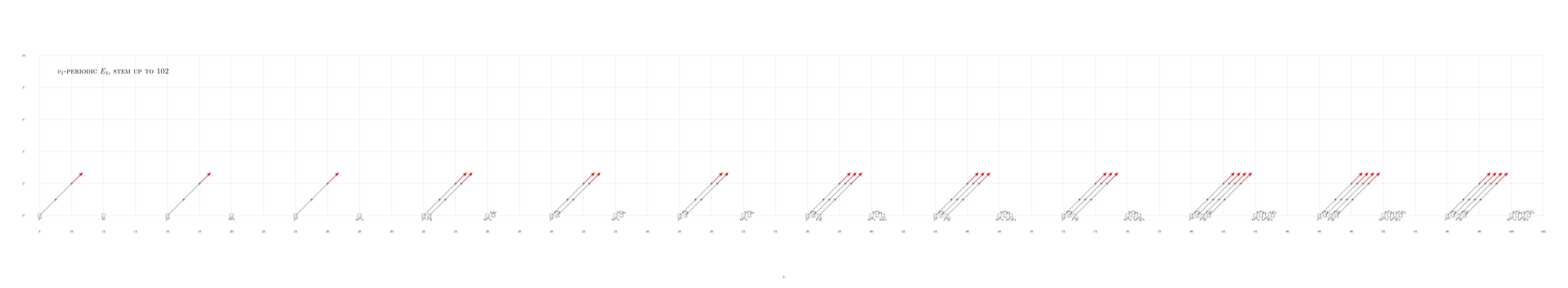}}
\caption{The $v_1$-periodic portion of the $\C$-motivic
Adams-Novikov $E_\infty$-page for $\mmf$}
\label{fig:E4v1}
\hfill
\end{center}
\end{figure}

\newpage
\KOMAoption{paper}{portrait, 18.2in:18.2in}
\KOMAoption{DIV}{last}
\newgeometry{margin=0in,footskip=0.3in}

\begin{figure}[H]
\begin{center}
{\includegraphics[trim={0cm, 0cm, 0cm, 0cm},clip,page=1,scale=0.2]{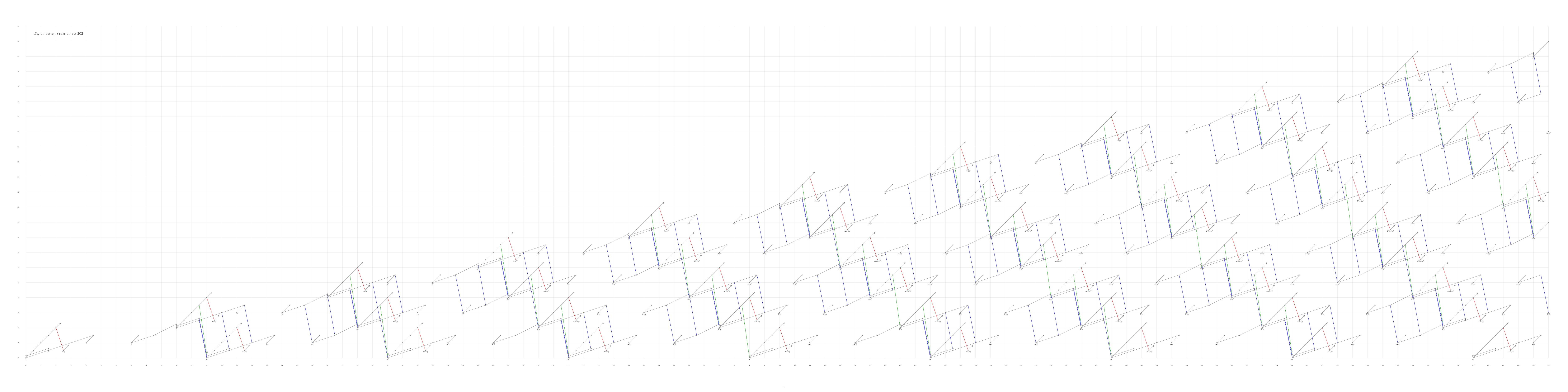}}
\caption{The $\C$-motivic Adams-Novikov $E_2$-page for $\mmf$
with differentials of length at most $7$}
\label{fig:E2d7}
\hfill
\end{center}
\end{figure}

\begin{figure}[H]
\begin{center}
{\includegraphics[trim={0cm, 0cm, 0cm, 0cm},clip,page=1,scale=0.2]{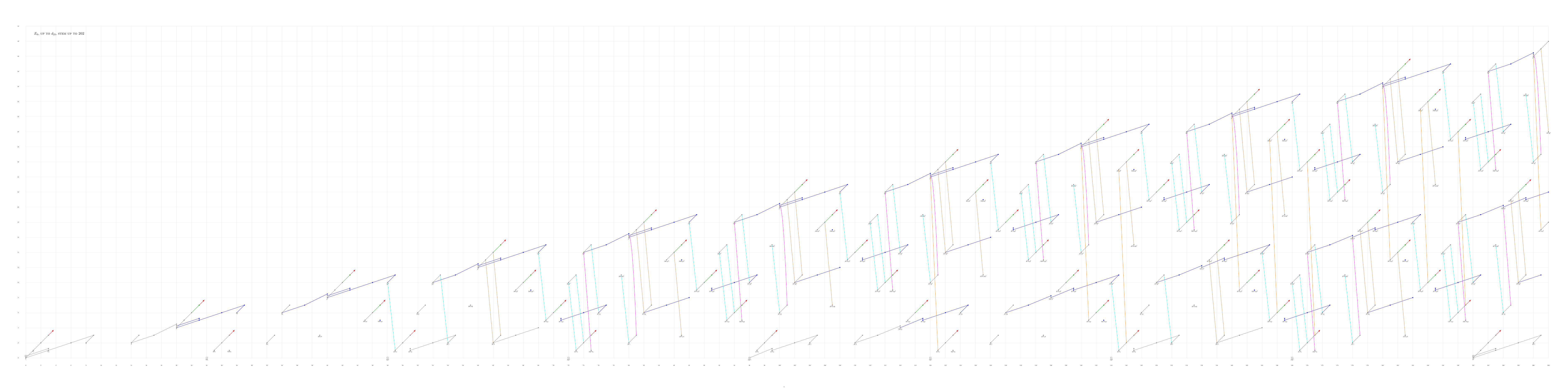}}
\caption{The $\C$-motivic Adams-Novikov $E_9$-page for $\mmf$
with differentials of length at least $9$}
\label{fig:E9d23}
\hfill
\end{center}
\end{figure}

\begin{figure}[H]
\begin{center}
{\includegraphics[trim={0cm, 0cm, 0cm, 0cm},clip,page=1,scale=0.2]{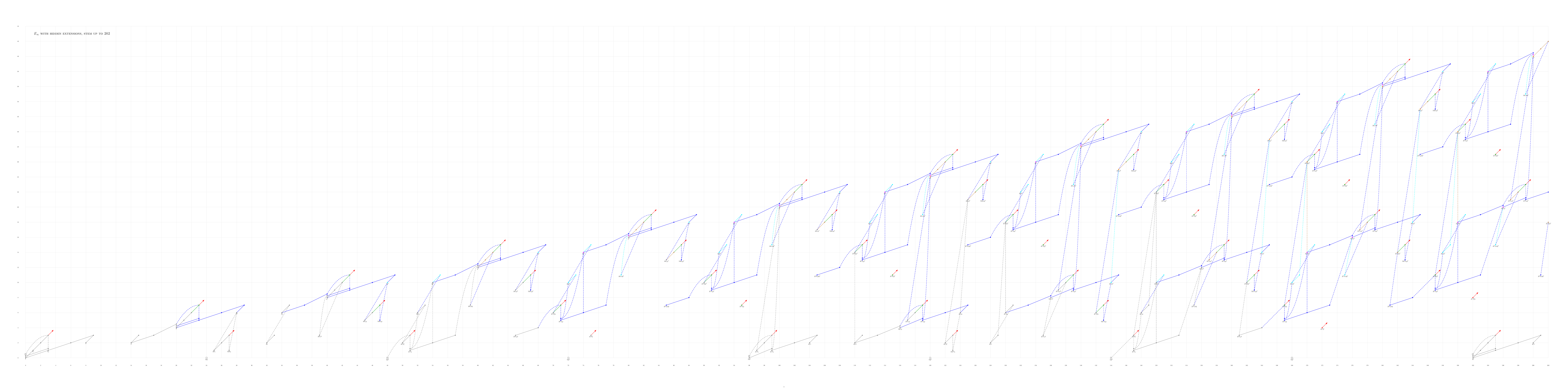}}
\caption{The $\C$-motivic Adams-Novikov $E_\infty$-page for $\mmf$
with hidden extensions by $2$, $\eta$, and $\nu$}
\label{fig:E-infty}
\hfill
\end{center}
\end{figure}

\newpage
\KOMAoption{paper}{portrait,8.5in:11in}
\KOMAoption{DIV}{last}
\newgeometry{margin=1in,footskip=0.5in}

	\bibliographystyle{alpha}
	\bibliography{bib}
	
\end{document}